\documentclass{birkjour}
\usepackage{graphicx}
\usepackage{graphics}
\usepackage{amssymb, amsmath, amsfonts}

 \newtheorem{theorem}{Theorem}[section]
 \newtheorem{corollary}[theorem]{Corollary}
 \newtheorem{lemma}[theorem]{Lemma}
 \newtheorem{proposition}[theorem]{Proposition}
 \theoremstyle{definition}
 \newtheorem{definition}[theorem]{Definition}
 \theoremstyle{remark}
 \newtheorem{remark}[theorem]{Remark}
 \newtheorem{example}[theorem]{Example}
 \numberwithin{equation}{section}

\newcommand{\R}{{\mathbb R}}
\newcommand{\oR}{{\overline{\mathbb R}}}
\newcommand{\N}{{\mathbb N}}

\newcommand{\To}{\longrightarrow}

\def\1{\^{\i}}
\def\2{\u{a}}
\def\3{\c{s}}
\def\4{\^{a}}
\def\5{\c{t}}
\def\a{\alpha}

\def\e{\epsilon}

\def\l{\lambda}
\def\<{\langle}
\def\>{\rangle}

\DeclareMathOperator*\dom{dom}
\DeclareMathOperator*\epi{epi}
\DeclareMathOperator*\cl{cl}
\DeclareMathOperator*\co{co}

\DeclareMathOperator*\pr{pr}

\DeclareMathOperator*\lin{lin}

\DeclareMathOperator*\inte{int}

\begin{document}
\title[Minimax Results on Dense Sets]% and Dense Family of Functionals in  $C(K)$ and $B(K)$ ]
 {Minimax Results on Dense Sets and Dense Families of Functionals in  $C(K)$ and $B(K)$} %and an Extension of James' Theorem}

%----------Author 1
\author{Szil\'ard L\' aszl\' o}

\address{Department of Mathematics\\ Technical University of Cluj-Napoca\\
              Str. Memorandumului nr. 28, 400114 Cluj-Napoca, Romania.\\Faculty of Mathematics and Computer Sciences\\ Babe\c s-Bolyai University,
400084 M. Kog\u alniceanu 1, Cluj-Napoca
Romania}
              \email{laszlosziszi@yahoo.com}

\thanks{This work was supported by a grant of the Romanian Ministry of Education, CNCS - UEFISCDI, project number PN-II-RU-PD-2012-3 -0166 and a grant from the Romanian National Authority for Scientific Research,
CNCS-UEFISCDI, project number PN-II-ID-PCE-2011-3-0024.}



%\date{January 1, 2004}
\begin{abstract}  Let $X$ and $Y$ be two arbitrary sets and let $f:X\times Y\To\R$ be a bifunction.  Recall that a minimax theorem deals with sufficient conditions under which the equality $\inf\nolimits_{x\in X}\sup\nolimits_{y\in Y} f(x, y) =\sup\nolimits_{y\in Y}\inf\nolimits_{x\in X} f(x, y)$
holds. Generally the set $X$ is assumed to be compact, ​minimax results on dense sets, (that is $X$ is  dense in a subset of a topological vector space), are absent in the literature. In this paper we give a motivation of this absence, we show by  an example that the general results of Fan and Sion cannot be extended on usual dense sets. Nevertheless,  we obtain some new minimax results on  special type of dense sets.
 We apply our results in order to obtain denseness of some family of functionals in the function spaces $C(K)$ and $B(K),$ respectively.
\end{abstract}
%%% ----------------------------------------------------------------------
\maketitle
%%% ----------------------------------------------------------------------
%\tableofcontents\section{Introduction}
\section{Introduction}

 In this paper we obtain some minimax results on dense sets. We recall, that among the most general minimax results are the ones obtained by Fan \cite{Fan} and Sion \cite{Si}, and both assume the compactness of $X.$  As a matter of ​fact, ​minimax results on dense sets  are absent in the literature. In this paper we give a motivation of this absence, Example \ref{extm} shows that the general results of Fan and Sion cannot be extended on usual dense sets. Nevertheless,  we obtain some new minimax results on a special type of dense set that we call self-segment-dense \cite{La,LaVi,LaVi1}. Moreover, under the strong  assumption of equicontinuity of the family $\{f(\cdot,y)\}_{y\in Y}$ we are able to obtain some minimax results on general dense sets. Our approach is based on some results that ensures that the infimum of a convex function on a dense set coincides with its global infimum.  In first step we provide conditions that assure the infimum of a convex function on a dense set of its domain coincide with the infimum on the convex hull of that dense set. As a second step we give conditions that assure the infimum on this convex hull is equal to the global infimum of the function. In the same manner, we provide some conditions that ensure the coincidence of two convex function that are equal on a dense subset of their domain. Then we apply these results in order to obtain some minimax results on dense sets.
 Several examples and counterexamples circumscribe our research and motivates our approach considering special type of dense sets.

 In what follows, for the convenience of the reader, we recall Fan's  minimax result (see \cite{Bor,Fan}).

\begin{theorem}\label{fan} Suppose that $X$ and $Y$ are non-empty sets and  let $f:X\times Y\to\R$ be a function  convexlike on $X$ and concavelike  on $Y$. Suppose that $X$ is compact and $x\To f(x,y)$ is lower semicontinuous on $X$ for each $y\in Y.$ Then,
$$\min\nolimits_{x\in X}\sup\nolimits_{y\in Y} f(x, y) =\sup\nolimits_{y\in Y}\min\nolimits_{x\in X} f(x, y).$$
\end{theorem}

Sion's minimax results holds under different assumptions (see \cite{Si}).
\begin{theorem}\label{sion} Let $X$ be a compact and convex subset of  topological vector space and let $Y$ be a convex subset of a topological vector space. Let $f:X\times Y\to\R$ be a function  and assume that $x\To f(x,y)$ is lower semicontinuous and quasi-convex on $X$ for each $y\in Y,$ and $y\To f(x,y)$ is upper semicontinuous and quasi-concave on $Y$ for each $x\in X.$ Then,
$$\min\nolimits_{x\in X}\sup\nolimits_{y\in Y} f(x, y) =\sup\nolimits_{y\in Y}\min\nolimits_{x\in X} f(x, y).$$
\end{theorem}

For a comprehensive survey in the field of minimax theorems we refer to \cite{S,S1}.

In this paper we obtain some Fan type minimax result, where the set $X$ is a (special) dense subset of a convex set. We also show that our result fails in general, it is not enough to  assume  only that $X$ is a dense subset of a compact and convex set.

The paper is organized as follows. In next section, we introduce some preliminary notions and the necessary apparatus that we need in order to obtain our results. We also present the notion of a self-segment-dense set introduced in \cite{LaVi} and  the notion of a segment-dense set introduced in \cite{DTL}. We show that these notions are incomparable, we give an example of a self-segment-dense set which is not segment-dense, and an example of a segment-dense set which is not self-segment-dense. Further an example of a dense set is given, which is neither self-segment-dense nor segment-dense.

Sections 3 deals with convex functions. We provide conditions that assure the infimum of a convex function on a dense subset of its domain is equal to the global infimum of that function. In Section 4 we apply these results in order to obtain some minimax results on dense sets. Also here, by an example  we show that the extension of Fan's and Sion's minimax result to usual dense sets is impossible. In the final section we apply our results in order to obtain denseness of some family of functionals in the  function spaces $C(K)$ and $B(K),$ respectively.
\section{Preliminaries}

\subsection{Convexity notions for real valued functions}

Let $X$ be a real Hausdorff, locally convex topological vector space.  For a non-empty set $D\subseteq X$, we denote by $\inte(D)$ its interior, by $\cl(D)$ its closure, by $\co(D)$ its convex hull, by $\overline{\co}( D)=\cl(\co( D))$ its closed convex hull and by $\lin (D)$ the subspace of $X$ generated by $D$.  We say that $P\subseteq D$ is dense in $D$ iff $D\subseteq \cl (P)$, and that $P\subseteq X$ is relatively compact iff $\cl(P)$ is compact.

We say that the function $f:X\To \oR=\R\cup\{\pm\infty\}$ is convex if
\begin{equation*}
\forall x,y\in X, \ \forall t\in[0,1] \ :f(t x+(1-t)y)\leq t f(x)+(1-t)f(y),
\end{equation*}
 with the conventions $(+\infty)+(-\infty)=+\infty,$ $0\cdot(+\infty)=+\infty$ and $0\cdot(-\infty)=0$. We consider $\dom f=\{x\in X:f(x)<+\infty\}$ the \emph{domain} of $f$ and $\epi f=\{(x,r)\in X\times\R:f(x)\leq r\}$ its \emph{epigraph}.  We call $f$ \emph{proper} if $\dom f\neq\emptyset$ and $f(x)>-\infty$ for all $x\in X$. By $\overline{f}$ we denote the \emph{lower semicontinuous hull} of $f$, namely the function whose epigraph is the closure of
$\epi f$ in $X\times\R$, that is $\epi(\overline{f})=\cl(\epi f).$

We say that the set $U\subseteq \dom f$ is graphically dense in $\dom f$, (see \cite{BD,P}), if for all $x\in\dom f$ there exists a net $(u_i)\subseteq U$ such that $u_i\To x$ and $f(u_i)\To f(x).$% Similarly, we say that the set $U\subseteq \dom f$ is semi-graphically dense in $\dom f$, if for all $x\in\dom f$ there exists a net $(u_i)\subseteq U$ such that $u_i\To x$ and $\liminf_{u_i\To x} f(u_i)= f(x).$

Let $f:X\to\oR$ be a proper function and consider the set  $U\subseteq \dom f$.
We define the set $\epi f_U$ as $$\epi f_U:=\{(u,r)\in U\times \R:f(u)\le r\},$$
and the function $\overline{f_U}:X\To\oR$ by
$$\epi\overline{f_U}=\cl(\epi f_U).$$

We say that $f$ is convex on $U$, iff for all $u,v\in U,\,t\in[0,1]$ such that $(1-t)u+tv\in U$ one has
$$f((1-t)u+tv)\le(1-t)f(u)+tf(v).$$
Note that we did not  assume the convexity of  $U$ in the previous definition.

We say that $f$ is lower semicontinuous at $x\in X$ if for every net $(x^i)\subseteq X$ converging to $x$ one has $\liminf_{x^i\To x}f(x^i)\ge f(x).$
$f$ is upper semicontinuous at $x$ iff $-f$ is lower semicontinuous at $x.$ $f$ is lower semicontinuous on $X$ if is lower semicontinuous at every point  of $X.$ Note that $f$ is lower semicontinuous on $X$ iff $\epi f$ is closed, that is $f=\overline{f}.$

We say that $f$ is quasiconvex, if its domain is convex and for all $x,y\in \dom f$ and $t\in[0,1]$ one has $f((1-t)x+ty)\le\max\{f(x),f(y)\}.$ $f$ is quasiconcave iff $-f$ is quasiconvex.

For two arbitrary sets $X$ and $Y$, and a bifunction $f:X\times Y\To\R$, Ky Fan \cite{Fan} introduced the following notions.

$f$ is convexlike on $X$, iff for all $x_1,x_2\in X,\,t\in[0,1]$ there exists $x_3\in X$ such that
$$f(x_3,y)\le(1-t)f(x_1,y)+tf(x_2,y)\,\forall y\in Y.$$
$f$ is concavelike on $Y$ iff for all $y_1,y_2\in Y,\,t\in[0,1]$ there exists $y_3\in Y$ such that
$$f(x,y_3)\ge(1-t)f(x,y_1)+tf(x,y_2)\,\forall x\in X.$$
Note that in these definitions no algebraic structure on $X$ and $Y$ are assumed. Obviously, if the mapping $x\To f(x,y)$ is convex for every $y\in Y$ then $f$ is convexlike in its first variable.

\subsection{On some remarkable properties of self-segment-dense sets}
Let $X$ be a real Hausdorff, locally convex topological vector space.  We will  use the following notations for the open, respectively
closed, line segments in $X$ with the endpoints $x$ and $y$
\begin{eqnarray*}
]x,y[ &:=&\big\{z\in X:z=x+t(y-x),\,t\in ]0,1[\big\}, \\
\lbrack x,y] &:=&\big\{z\in X:z=x+t(y-x),\,t\in \lbrack 0,1]\big\}.
\end{eqnarray*}
The line segments $]x,y],$ respectively $[x,y[$ are defined similarly.
In \cite{DTL}, Definition 3.4, The Luc has introduced the notion of a so-called \emph{segment-dense} set. Let $V\subseteq X$ be a convex set. One
says that the set $U\subseteq V$ is segment-dense in $V$ if for each $x\in V$ there can be found $y\in U$ such that $x$ is a cluster point of the set $%
[x,y]\cap U.$

In what follows we present a denseness notion (cf. \cite{LaVi,LaVi1}) which is slightly different from the concept of The Luc presented above, but which is in some sense compatible with the convexity property of sets.

\begin{definition}
\label{dd} Consider the sets $U\subseteq V\subseteq X$ and assume that $V$
is convex.
We say that $U$ is self-segment-dense in $V$ if $U$ is dense in $V$ and
\begin{equation*}
\forall x,y\in U,\mbox{  the set }\lbrack x,y]\cap U\mbox{  is dense in }%
\lbrack x,y].
\end{equation*}
\end{definition}

\begin{remark}\rm\textrm{ Obviously in one dimension the concepts of a segment-dense set respectively a self-segment-dense set are equivalent to the concept of a dense set. }
\textrm{\ }
\end{remark}

In what follows we provide an essential example of a self-segment-dense set.

\begin{example}\rm [see also \cite{LaVi1}, Example 2.1] Let $V$ be the real Hilbert space of square summable sequences $l_2$ and define $U$ to be the set
\begin{equation*}
U :=\{(x)=(x_1,...,x_n,...) \in l_2 : x_i\in \mathbb{Q},\mbox{ for all } i\in \N\},
\end{equation*}
where $\mathbb{Q}$ denotes the set of all rational numbers. Then, it is
clear that $U$ is dense in $l_2.$ On the other hand $U$ is not
segment-dense in $l_2,$ since for $(x)=\left(\sqrt{2},\frac12,...,\frac1n,...\right)\in {l_2} $ and for every $(y)=(y_1,y_2,...,y_n,...)\in U$, one has $[(x), (y)] \cap U = \{(y)\}.$

It can easily be observed that $U$ is self-segment-dense in $l_2$, since for every $(x),(y)\in U,$ $(x)=(x_1,...,x_n,...),\,(y)=(y_1,...,y_n,...)$ we have $[(x),(y)]\cap U=\{(x_1+t(y_1-x_1),...,x_n+t(y_n-x_n),...): t\in[0,1]\cap\mathbb{Q}\},$ which is
obviously dense in $[(x),(y)].$
\end{example}

Another very interesting example comes from the general infinite dimensional setting.  Let $X$ be a nonreflexive Banach space and let $B$ be the closed unit ball of $X$. Let $X^{**}$ be the bidual space of $X$ and let $B^{**}$ be its closed unit ball. According to Goldstine Theorem \cite{fabian} in this case $B$ is dense in $B^{**}$ in the weak$^*$ topology of $X^{**}.$ Moreover, in virtue of convexity of $B$, we have that $B$ is actually self-segment-dense in $B^{**}$ with respect to the weak$^*$ topology of $X^{**}.$

To further circumscribe the notion of a self-segment-dense set we provide an example %
of a subset that is dense but not self-segment-dense.

\begin{example}\rm\label{ex1}
Let $X$ be an infinite dimensional real Hilbert space. It is known that the unit sphere %
$U=\left\{ x\in X:\left\Vert x\right\Vert =1\right\},$
is dense with respect to the weak topology in the unit ball $B=\left\{ x\in X:\left\Vert x\right\Vert \leq1 \right\}$, but %
it is obviously not self-segment-dense since any segment with endpoints on the sphere does not intersect the sphere in any other points. Moreover, $U$ is not segment-dense in the sense of The Luc either, because for every $x\in U$ one has $[0,x]\cap U=\{0\}.$ Note that the same argument is also valid if $X$ is a  normed space. In this case one can take two points $x$ and $-x$ belonging to the unit sphere, and show that the intersection of the segment $[-x,x]$ with the unit sphere is $\{-x,x\}.$ Moreover, if $X$ is not strictly convex then the unit sphere contains segments.
\end{example}

In what follows we  provide an example of a dense set in the strong topology of a Hilbert space, which is also segment-dense but is not self-segment-dense (see also Proposition 3.4, \cite{BD}).

\begin{example}\rm\label{ex11} Let $V=X=l_2$ be the real Hilbert space of square summable sequences and define $U$ to be the set
\begin{equation*}
U:=\{(x)=(x_1,...,x_n,...) \in l_2 : |x_1|<\sup_{n\ge 2}n|x_n|\}.
\end{equation*}
Then $U$ is open, dense and segment-dense in $l_2$, but is not self-segment-dense.
\end{example}
\begin{proof}
First of all observe that the complement of $U$ is closed. Indeed, for a sequence $(x^i)=(x_1^i,...,x_n^i,...)\in l_2\setminus U$ converging to $(x)=(x_1,...,x_n,...)\in l_2$, one has $|x_1^i|\ge n|x_n^i|$ for all $n\in\N$. Since $x_n^i\To x_n,\,i\To\infty$ for all $n\in\N$, obviously
$|x_1|\ge n|x_n|$ for all $n\in\N$. Hence, $l_2\setminus U$ is closed, which shows that $U$ is open. Consider now $(x)=(x_1,...,x_n,...)\in l_2\setminus U.$ We show that for every $\e>0$ there exists $(y)\in U$ such that $\|(x)-(y)\|<\e.$ Indeed, let $\e>0$ and consider $n_0\in\N$ such that $n_0>\frac{3|x_1|}{\e}.$ %Since $(x)\in l_2\setminus U$ one has $\e>\frac{3|x_1|-3n_0|x_{n_0}|}{n_0}\ge 0$, hence $n_0(|x_{n_0}|+\frac\e3)> |x_1|.$
Let $(y)=\left(y_1,y_2,...,y_n,...\right)\in l_2, y_n=x_n$ for all $n\neq n_0$ and $y_{n_0}=|x_{n_0}|+\frac\e3.$ Then $n_0|y_{n_0}|>|y_1|$, hence $(y)\in U$.

On the other hand $$\|(y)-(x)\|=|\frac\e3+|x_{n_0}|-x_{n_0}|\le2|x_{n_0}|+\frac\e3\le\frac{2|x_1|}{n_0}+\frac\e3<\frac{2\e}{3}+\frac\e3=\e.$$

Next, we show that $U$ is not self-segment-dense in $l_2.$ Indeed, consider $(x)=(1,1,\frac19,...,\frac{1}{n^2},....)$ and $(y) = (1,-1,\frac{1}{18},...,\frac{1}{2n^2},...)\in U.$ We show that $(1-t)(x)+t(y)\cap U=\emptyset$ for all $t\in\left[\frac14,\frac34\right].$ We have $(1-t)(x)+t(y)=\left(1,1-2t,\frac{2(1-t)+t}{18},...,\frac{2(1-t)+t}{2n^2},...\right),$ hence $(1-t)(x)+t(y)\in U,$ if and only if $1<2|1-2t|.$ But then $t\not\in \left[\frac14,\frac34\right].$ Consequently $\cl([(x),(y)]\cap U)\neq[(x),(y)],$ which shows that $U$ is not self-segment-dense in $l_2.$

It remained to show, that $U$ is segment-dense in $l_2.$ Actually we will show something more, that is, for every $(y)\in l_2$ there exists $(x)\in U$ such that $[(x),(y)[\subseteq U.$ When $(y)\in U$ the statement follows from the fact that $U$ is open. Let now $(y)=(y_1,y_2,...,y_n,..)\in l_2\setminus U.$
Then $|y_1|\ge n|y_n|$ for all $n\in\N,$ hence one has $-\frac{|y_1|}{n}\le y_n\le\frac{|y_1|}{n}$ for all $n\in\N.$ Consider $(x)=\left(y_1,\frac{\sqrt{\ln 2}}{2},...,\frac{\sqrt{\ln n}}{n},...\right).$ By using the integral test, one easily can show that $(x)\in l_2.$ Since $\sup_{n\ge 2}n|x_n|=\infty>|y_1|$ we get that $(x)\in U.$ For $t\in]0,1]$ we show that $(1-t)(y)+t(x)\in U,$ that is $[x,y[\subseteq U.$ Indeed, one has
$$(1-t)(y)+t(x)=\left(y_1,(1-t)y_2+t\frac{\sqrt{\ln 2}}{2},...,(1-t)y_n+t\frac{\sqrt{\ln n}}{n},...\right).$$

Since for $n\in\N$ big enough and for $t>0$ fixed, one has $$n\left|(1-t)y_n+t\frac{\sqrt{\ln n}}{n}\right|\ge n\left|-(1-t)\frac{|y_1|}{n}+t\frac{\sqrt{\ln n}}{n}\right|\ge t\sqrt{\ln n}-(1-t)|y_1|,$$ we get that $$\sup_{n\in \N}n\left|(1-t)y_n+t\frac{\sqrt{\ln n}}{n}\right| =\infty.$$ Hence, $(1-t)(y)+t(x)\in U$ for all $t\in]0,1].$
\end{proof}

\begin{remark}\rm
Note that every  convex subset of a topological vector space is self-segment-dense in its closure.  In particular dense subspaces and dense affine subsets are
self-segment-dense. Therefore if $U\subseteq V$ is dense in $V$ and $V$ is convex, then $\co(U)$ is self-segment-dense in $V.$ %Even more, the set
%$$\bigcup_{u_1,u_2\in U}[u_1,u_2]$$
%is self-segment-dense in $V.$}
\end{remark}

Next we provide some remarkable properties of a self-segment-dense set. We also show that these results do not hold if we replace the self-segment-dense property of the set involved by its denseness.

The following lemma (see \cite{LaVi1}, \cite{AR}) gives an interesting characterization of self-segment-dense sets and will be used in the sequel.

\begin{lemma}\label{l25}\rm[{Lemma 2.1, \cite{LaVi1}}] Let $X$ be a Hausdorff locally convex topological vector space,
let $V\subseteq X$ be a convex set and let $U\subseteq V$ a self-segment-dense set in $V.$ Then, for all finite subset $\{u_1,u_2,\ldots,u_n\}%
\subseteq U$ one has $$\cl(\co\{u_1,u_2,\ldots,u_n\}\cap U)=\co\{u_1,u_2,\ldots,u_n\}.$$
\end{lemma}
\begin{remark}\rm\label{r28} Observe that under the hypothesis of Lemma \ref{l25}, one has that the intersection $\co\{u_1,u_2,\ldots,u_n\}\cap U$ is self-segment-dense in $\co\{u_1,u_2,\ldots,u_n\}.$
\end{remark}
Let us emphasize that this result does not remain valid in case we replace the self-segment-denseness of $U$ in $V$, by its denseness in $V,$ as the next example shows.
\begin{example}\rm[Example 2.3, \cite{La}] Let $V$ be the closed unit ball of an infinite dimensional Banach space $X$, and let $x,y\in V, \, x\neq y.$ Moreover, consider $u,v\in ]x,y[,$ $u=x+t_1(y-x),$ $v=x+t_2(y-x),$ with $t_1,t_2\in]0,1[,\,t_1<t_2.$ Then obviously $U=V\setminus [u,v]$ is dense in $V$, but not self-segment-dense, since for $x,y\in U$ the set  $[x,y]\cap U=[x,u[\,\cup\, ]v,y]$ is not dense in $[x,y].$ This also shows, that $\cl(\co\{x,y\}\cap U)\neq\co\{x,y\}.$
\end{example}

An easy consequence of Lemma \ref{l25} is the following more general result.
\begin{lemma}\label{l26} Let $X$ be a Hausdorff locally convex topological vector space,
let $V\subseteq X$ be a convex set and let $U\subseteq V$ a self-segment-dense set in $V.$ Then, for every subset $\mathcal{S}\subseteq U$ one has $$\cl(\co(\mathcal{S})\cap U)=\overline{\co}(\mathcal{S}).$$
In other words, $\co(\mathcal{S})\cap U$ is self-segment-dense in $\co(\mathcal{S}).$
\end{lemma}
\begin{proof} Let $x\in \co(\mathcal{S}).$ We show that for every neighbourhood $G$ of $x$ one has $G\cap U\neq\emptyset.$ Indeed, $x=\sum_{i=1}^n \l_iu_i$ for some $u_i\in S\subseteq U,\l_i\ge 0,\,i\in\{1,2,\ldots,n\},\,n\in\N\,,\sum_{i=1}^n\l_i=1.$ Hence, $x\in\co\{u_1,u_2,\ldots,u_n\}$ and according to Lemma \ref{l25}, $\co\{u_1,u_2,\ldots,u_n\}\cap U=\co\{u_1,u_2,\ldots,u_n\}\cap S$ is dense in $\co\{u_1,u_2,\ldots,u_n\}.$ Consequently $G\cap U\neq\emptyset$ which shows that $\co(\mathcal{S})\cap U$ is dense in $\co\mathcal{S}.$ It is also self-segment-dense since for $u_1,u_2\in \co(\mathcal{S})\cap U$ one has
$$\cl([u_1,u_2]\cap\co(\mathcal{S})\cap U)=\cl([u_1,u_2]\cap U)=[u_1,u_2].$$
\end{proof}

\begin{remark}\rm\label{r210} Obviously, the precedent lemma also ensures that ${\co}(\mathcal{S})\cap U$ is self-segment-dense in $\overline{\co}(\mathcal{S}).$ A particular instance is, that $\cl( U)=\overline{\co}(U).$
\end{remark}

%Before we are able to prove Theorem \ref{tm} we need some preliminary results.
In what follows we present a simple but very useful result concerning on self-segment-dense sets, by showing  that the self-segment-dense property of a subset of some base set implies the convexity of the base set.

\begin{lemma}\label{l2.6} Let $V\subseteq X$ be closed and let $U\subseteq V$ be dense in $V,$ with the property that for all $u,v\in U$ one has that $[u,v]\cap U$ is dense in $[u,v].$ Then $V$ is convex, hence $U$ is actually self-segment-dense in $V.$
\end{lemma}
\begin{proof} Observe that for all $u,v\in U$ one has $[u,v]\subseteq V$ since
$$[u,v]=\cl([u,v]\cap U)\subseteq\cl(U)=V.$$
Assume now that $V$ is not convex, i.e., there exist $x,y\in V,$ and $t_0\in(0,1)$ such that $(1-t_0) x+t_0y=z_0\not\in V.$ Since $U$ is dense in $V$, there exists the nets $(x_i),(y_i)\subseteq U$  such that $x_i\To x,\,y_i\To y.$ But then $[x_i,y_i]\subseteq V$ for all $i$, hence $(1-t_0)x_i+t_0y_i\in V$ for all $i$. Since $V$ is closed we obtain that $\lim((1-t_0)x_i+t_0y_i)=z_0\in V,$ contradiction.
\end{proof}
\begin{remark}\rm Note that the closedness of $V$ in the hypothesis of Lemma \ref{l2.6} is essential. Indeed, let $A\subseteq \R^2$ be the
 square with vertices $(-1,-1),(-1,1),(1,1)$ and $(1,-1),$ let $V$ be $A\setminus](-1,-1),(1,-1)[$ and let $U$ be the interior of $A.$  Here, the interior of the square is meant relative to the plane of the square. Then obviously $U$ has the property, that for all $u,v\in U$ one has that $[u,v]\cap U$ is dense in $[u,v].$ Moreover, $U$ is dense in $V$.  Observe that $V$ is not convex, because $(-1,-1),(1,-1)\in V$ but $[(-1,-1),(1,-1)]\not\subseteq V.$ This is due to the fact that $V$ is not closed, hence Lemma \ref{l2.6} cannot be applied.
\end{remark}

\section{Convex functions and dense sets}

In what follows we provide some  results concerning convex functions on dense sets. Our results are based on the concepts of a self-segment-dense and a segment-dense set, respectively. We obtain conditions that ensure the coincidence of two proper, convex  and lower semicontinuous function that are equal on a dense subset of their domain. Further, we analyze the situation when the infimum  of a convex function is equal to the infimum of that function taken over a dense subset of its domain. An example shows that the use of  the concepts of  a self-segment-dense set and segment-dense set are essential in order to obtain these results. %These results  are essential in order to prove our minimax results in the  next section.
\begin{theorem}\label{t2.8} Let $f:X\To\oR$ be a proper function and let $U\subseteq \dom f$ be a dense set.
\begin{itemize}
  \item[(a)] Then, $\overline{f_U}$  is lower semicontinuous, $U$ is graphically dense in $\dom \overline{f_U}$ and  $\inf_{x\in U}f(x)=\inf_{x\in X}\overline{f_U}(x).$  %Further, $\overline{f_U}=f$ on $X$ if and only if $U$ is graphically dense in $\dom f.$
  \item[(b)] Assume that $f$ is  lower semicontinuous on $U.$ Then, $f(u)=\overline{f_U}(u)$ for all $u\in U$. Moreover, if $f$ is lower semicontinuous on $X$ then $\dom\overline{f_U}\subseteq\dom f$  and  $\overline{f_U}(x)\ge f(x)$ for all $x\in\ X.$  Further, in this case, $\overline{f_U}=f$ on $X$, if and only if, $U$ is graphically dense in $\dom f.$
  \item[(c)] Assume that $\dom f$ is convex, $U$ is self-segment-dense in $\dom f$ and $f$ is convex on $U$. Then, $\overline{f_U}$  is   also convex and and $\epi f_U$ is self-segment-dense in $\epi\overline{f_U}.$
\end{itemize}
\end{theorem}
\begin{proof}
$(a)$ The lower semicontinuity of $\overline{f_U}$ follows from the closedness of $\epi\overline{f_U}$.
We show that $U$ is graphically dense in $\dom \overline{f_U}.$ Let $(u_i,r_i)\subseteq \epi f_U,$ a net converging to $(x,\overline{f_U}(x))\in \epi \overline{f_U}.$ Then, $(u_i,r_i)\subseteq \epi \overline{f_U},$ hence $r_i\ge \overline{f_U}(u_i)$. Consequently $\overline{f_U}(x)=\lim r_i\ge\liminf \overline{f_U}(u_i)\ge \overline{f_U}(x),$ which shows that $\liminf \overline{f_U}(u_i)= \overline{f_U}(x).$ (Note that we have also shown that
$\liminf {f}(u_i)= \overline{f_U}(x).$) Hence, there exists a subnet of  the net $(\overline{f_U}(u_i))$, say $(\overline{f_U}(u_j))$ which converges to $\overline{f_U}(x).$ Since $(u_i)$ converges to $x$, obviously $\lim u_j=x$, which completes the proof. (One may similarly show, that there exists a net $(u_j)\subseteq U$ such that $\lim u_j=x$ and $\lim f(u_j)=\overline{f_U}(x).$)

 We show next, that $\inf_{x\in U}f(x)=\inf_{x\in X}\overline{f_U}(x).$ Note at first that \\$\inf_{x\in U}f(x)\ge\inf_{x\in X}\overline{f_U}(x).$ Indeed, for every $x\in U$ one has $(x,f(x))\in\epi f_U\subseteq \epi\overline{f_U},$ hence $\overline{f_U}(x)\le f(x)$ for all $x\in U.$ Consequently, $$\inf\nolimits_{x\in X}\overline{f_U}(x)\le \inf\nolimits_{x\in U}\overline{f_U}(x)\le \inf\nolimits_{x\in U}f(x).$$
 If $\inf\nolimits_{x\in U}f(x)=-\infty$ then the statements is obvious.  Assume now, that \\$\inf_{x\in U}f(x)=\a>-\infty.$ Then, $f(x)\ge \a$ for all $x\in U.$ Let $x\in \dom \overline{f_U}.$ Then, according to the previous part of proof, there exists a net $(u_i)\subseteq U$, such that $\overline{f_U}(x)=\lim f(u_i)\ge \a,$ hence $\inf_{x\in X}\overline{f_U}(x)\ge \a.$ This also shows that, in this case, $\overline{f_U}$ is proper.

$(b)$ We show   that $\overline{f_U}(u)=f(u)$ for all $u\in U.$ Let $u\in U.$ Since $(u,f(u))\in\epi f_U\subseteq\epi\overline{f_U}$ one has $$\overline{f_U}(u)\le f(u).$$
Let $r\in\R$ such that  $(u,r)\in\epi\overline{f_U}.$ Then, there exists a net $((u_i,r_i))\subseteq\epi f_U$ such that $(u_i,r_i)\To(u,r).$ Obviously $r_i\ge f(u_i)$ and in virtue of lower semicontinuity of $f$ on $U$ we have
$$r=\lim r_i=\liminf r_i\ge\liminf f(u_i)\ge f(u).$$
Sice $r$ is a arbitrary, provided $r\ge\overline{f_U}(u)$, one has
$$\overline{f_U}(u)\ge f(u).$$

Assume now that $f$ is lower semicontinuous on $X$. Obviously, $\epi f$ is closed, hence $\epi\overline{f_U}=\cl(\epi f_U)\subseteq\epi f.$ The latter relation leads to $$\dom\overline{f_U}=\pr\nolimits_X(\epi\overline{f_U})\subseteq\pr\nolimits_X(\epi f)=\dom f.$$ Consequently, for all $x\in\dom\overline{f_U}$ one has $(x,\overline{f_U}(x))\in\epi\overline{f_U}\subseteq\epi f$ which leads to $f(x)\le\overline{f_U}(x),$ for all $x\in\dom\overline{f_U}$. Since $\overline{f_U}(x)=+\infty$ for all $x\in X\setminus \dom \overline{f_U}$ we get that $f(x)\le\overline{f_U}(x),$ for all $x\in X.$ This also shows that, in this case, $\overline{f_U}$ is proper.

It  remained to show, that $\overline{f_U}=f$ on $X,$ if and only if $U$ is graphically dense in $\dom f.$ Assume that $\overline{f_U}=f$ on $X$ and let $x\in \dom f.$ Then, according to $(a)$, there exists a net $(u_j)\subseteq U$ such that $\lim u_j=x$ and $\lim \overline{f_U}(u_j)=\overline{f_U}(x).$ But $\overline{f_U}(u_j)=f(u_j)$ and $\overline{f_U}(x)=f(x).$ Conversely, consider a net $(u_j)\subseteq U$ such that $\lim u_j=x$ and $\lim f(u_j)=f(x).$ Note that $\overline{f_U}(x)\le \lim f(u_j)=f(x).$ On the other hand, $\overline{f_U}(x)\ge f(x)$, which completes the proof.

$(c)$ Assume now that $\dom f$ is convex, $U$ is self-segment-dense in $\dom f$ and that $f$ is convex on $U.$ Observe that by definition we have that $\epi f_U$ is dense in $\epi\overline{f_U}$ and $\epi\overline{f_U}$ is closed.

 Note that $\overline{f_U}$ is convex, if and only if $\epi\overline{f_U}$ is convex. For showing the convexity of $\epi\overline{f_U}$ we use Lemma \ref{l2.6}. To this end, we show that for all $(u,r),(v,s)\in \epi f_U$ we have that $$\cl([(u,r),(v,s)]\cap \epi f_U)=[(u,r),(v,s)].$$
Indeed, assume the contrary, i.e.,  there exist $(u,r),(v,s)\in \epi f_U$ such that $\cl([(u,r),(v,s)]\cap \epi f_U)\neq[(u,r),(v,s)].$ Then, there exist two distinct points $(u_1,r_1),(v_1,s_1)\in [(u,r),(v,s)]$ such that $$](u_1,r_1),(v_1,s_1)[\cap\epi f_U=\emptyset.$$
Since $u_1,v_1\in[u,v]$ and $U$ is self-segment-dense in $\dom f$, one has that $]u_1,v_1[\cap U\neq\emptyset,$ hence there exists $t_0\in]0,1[$ such that $$(1-t_0)u_1+t_0v_1=u_0\in]u_1,v_1[\cap U.$$ By the convexity of $f$ on $U$ we have $$f(u_0)\le(1-t_0)f(u_1)+t_0f(v_1)\le (1-t_0)r+t_0s,$$ which leads to $$(u_0,(1-t_0)r+t_0s)\in\epi f_U.$$
The latter relation shows that $(1-t_0)(u_1,r)+t_0(v_1,s)\in\epi f_U$ which contradicts the assumption that $](u_1,r_1),(v_1,s_1)[\cap\epi f_U=\emptyset.$ Thus, according to Lemma \ref{l2.6}  $\epi\overline{f_U}$ is convex and $\epi f_U$ is self-segment-dense in $\epi\overline{f_U}.$
 \end{proof}

\begin{remark}\rm\label{r33} Note that $\inf_{x\in U}f(x)=\inf_{x\in X}\overline{f_U}(x)$ implies that  $\overline{f_U}$ is  proper, provided $\inf_{x\in U}f(x)\neq-\infty$. The convexity of $\overline{f_U}$ is not guaranteed if in the hypothesis of Theorem \ref{t2.8} $(c)$ we   assume only that $U$ is dense in $\dom f,$ as the next example shows.
\end{remark}
\begin{example}\rm\label{ex2} Let $X$ be an infinite dimensional real Hilbert space. According to Example \ref{ex1},  the union of the unit sphere with $\{0\}$,  is dense with respect to the weak topology in the unit ball $B=\left\{ x\in X:\left\Vert x\right\Vert \leq1 \right\}$ of $X,$ but this set is not self-segment-dense in $B.$ Hence, let $U=\left\{ x\in X:\left\Vert x\right\Vert =1\right\}\cup\{0\}.$  Let $y_0\in U,\,y_0\neq 0,$ and consider the function $$f:X\To\oR,\,f(x)=\left\{\begin{array}{ll}\<y_0,x\>^3,\mbox{ if }x\in B\\+\infty,\mbox{ otherwise.}\\ \end{array}\right.$$

 Then, trivially, $f$ is convex on $U,$ since for $x,0,-x\in U$ one has $f(0)=\frac12 f(-x)+\frac12 f(x),$ but $\overline{f_U}$ is not convex on $X$. Indeed, it can easily be verified that $\overline{f_U}(y_0)=1,\,\overline{f_U}(-y_0)=-1$ and $\overline{f_U}\left(-\frac12 y_0\right)=-\frac18.$ Obviously $-\frac12 y_0=\left(1-\frac14\right)(-y_0)+\frac14 y_0\in[-y_0,y_0]$, but $$-\frac18=\overline{f_U}\left(\left(1-\frac14\right)(-y_0)+\frac14 y_0\right)>\left(1-\frac14\right)\overline{f_U}(-y_0)+\frac14\overline{f_U}(y_0)=-\frac12.$$
\end{example}
\begin{corollary}\label{c1}  Let $f:X\to\oR$ be a proper and continuous function and let $U\subseteq \dom f$ be a dense set. Then, $\inf_{x\in U}f(x)=\inf_{x\in X} f(x).$
\end{corollary}
\begin{proof} According to Theorem \ref{t2.8} $(a),$ $\inf_{x\in U}f(x)=\inf_{x\in X}\overline{f_U}(x).$ Since $f$ is continuous, obviously it is lower semicontinuous on $X$ and $U$ is graphically dense in $\dom f.$ Hence, by Theorem \ref{t2.8} $(b),$ one has $\overline{f_U}=f$ on $X.$
\end{proof}
However, Example \ref{exp39} below, shows that the continuity assumption of $f$ in the hypothesis of Corollary \ref{c1} is essential and cannot be replaced by lower semicontinuity.
\begin{remark}\rm\label{r36}  According to Theorem \ref{t2.8}, for any proper, lower semicontinuous and convex function $f$, and any dense set $U\subseteq \dom f$ there exists a proper and lower semicontinuous function $\overline{f_U}$ which coincides with $f$ on $U$, dominates $f$ on $X$ and is not equal to $f$, provided $U$ is not graphically dense in $\dom f.$ Moreover, if $U$ is also self-segment-dense, then $\overline{f_U}$ is convex. The general problem, that if two proper, convex and lower semicontinuous functions are equal on a dense set whether they coincide, has been investigated by Benoist and Daniilidis in \cite{BD}, in a Banach space context. According to Proposition 3.4 \cite{BD}, in infinite dimensions the answer is negative, which also follows from our  argument  above. Nevertheless, in finite dimension the answer is affirmative, as follows from Corollary 3.7 \cite{BD}. %In our read Corollary 3.7 \cite{BD} becomes the following.
However, there are special type of dense sets in infinite dimensional Hausdorff topological vector spaces, where the coincidence result holds. We show next, that if two proper, convex and lower semicontinuous functions are equal on a segment-dense set of their common domain, then they are equal everywhere.
\end{remark}

\begin{proposition}\label{p36} Let $X$ be a   Hausdorff locally convex topological vector space. Let $f,g:X\to\oR$ be two proper, convex and lower semicontinuous functions. Let $U\subseteq \dom f\cap \dom g$ be a segment-dense set such that $f(u)=g(u)$ for all $u\in U.$ Then, $f=g$ on $\dom f\cap \dom g.$
\end{proposition}
\begin{proof} Let $x_0\in \dom f\cap \dom g.$  Since $U$ is segment-dense in $\dom f\cap \dom g$, there exists $u_0\in U$ such that $x_0$ is a cluster point of the set $[u_0,x_0]\cap U.$ We show that $f(x_0)=g(x_0).$ Indeed, by the convexity and lower semicontinuity of $f$ and $g$, (in virtue Proposition 1.3.4, \cite{NP}) of one has that for every sequence $(x_n)\subseteq [u_0,x_0]$ converging to $x_0$, it holds $$f(x_0)=\lim_{x_n\To x_0}f(x_n)\mbox{  and }g(x_0)=\lim_{x_n\To x_0}g(x_n).$$
But, then for $(u_n)\subseteq [u_0,x_0]\cap U$ converging to $x_0,$ we get
$$f(x_0)=\lim_{u_n\To x_0}f(u_n)=\lim_{u_n\To x_0}g(u_n)=g(x_0).$$
\end{proof}

\begin{corollary}\label{c3} Let $X$ be a   Hausdorff locally convex topological vector space, let $f:X\to\oR$ be a proper, convex and lower semicontinuous function and let $U\subseteq \dom f$ be a segment-dense set in $\dom f.$ Then, $f=\overline{f_U}$ on $X$ and $\inf_{x\in U}f(x)=\inf_{x\in X}f(x).$
\end{corollary}
\begin{proof}  Using the same arguments as in the proof of Proposition \ref{p36}, one can easily show that  $U$ is graphically dense in $\dom f.$ Hence, according to Theorem \ref{t2.8} (b), $f=\overline{f_U}$ on $X.$ But Theorem \ref{t2.8} also shows that $\inf_{x\in U}f(x)=\inf_{x\in X}\overline{f_U}(x).$
%The statement $f=\overline{f_U}$ on $\dom \overline{f_U}$ follows by Theorem \ref{t2.8} and by taking $g=\overline{f_U}$ in Proposition \ref{p36}. %From  Theorem \ref{t2.8} we have $\inf_{x\in U}f(x)=\inf_{x\in \dom\overline{f_U}}f(x).$
%Assume that $\inf_{x\in U}f(x)>\inf_{x\in X}f(x).$ Then, there exists $x_1\in \dom f$ such that $f(x_1)<\a\le \inf_{x\in U}f(x).$ Since $U$ is segment-dense in $\dom f\cap \dom g$, there exists $u_0\in U$ such that $x_1$ is a cluster point of the set $[u_0,x_1]\cap U,$ consequently using the same arguments as in the proof of Proposition  \ref{p36}, we conclude that for every sequence $(u_n)\subseteq [u_0,x_1]\cap U$ converging to $x_1,$ we have
%$f(x_1)=\lim_{u_n\To x_1}f(u_n).$ But $\lim_{u_n\To x_1}f(u_n)\ge \a,$ contradiction.
\end{proof}

\begin{remark}\rm Note that the conclusion of Corollary \ref{c3} fails if we assume that $f:X\to\oR$ is convex only on $U.$ Moreover  $\dom \overline{f_U}$ might not be closed even when $\dom f$ is compact and $f$ is lower semicontinuous as the following simple example shows.
\end{remark}
\begin{example}\rm Consider the function $$f:\R\To\oR,\,f(x)=\left\{\begin{array}{lll} \frac{1}{x},\mbox{ if }x\in ]0,1]\\0,\mbox{ if }x=0\\+\infty,\mbox{ otherwise.}\\ \end{array}\right.$$
Then, $\dom f=[0,1]$ is compact and $f$ is convex on $U=]0,1]$ and  lower semicontinuous  on $\dom f.$ Further $U$ is segment-dense (and also self-segment-dense), in $\dom f.$ But, $1=\inf_{x\in U}f(x)\neq\inf_{x\in X}f(x)=0.$
On the other hand, $\overline{f_U}(x)=\left\{\begin{array}{ll} \frac{1}{x},\mbox{ if }x\in ]0,1]\\+\infty,\mbox{ otherwise}\\ \end{array}\right.$, hence $f\neq \overline{f_U}$ on $\R.$

It can easily be observed that $\dom\overline{f_U}=]0,1]$ which is not closed.
\end{example}
Next we provide a general  coincidence result, involving self-segment-dense sets, in infinite dimension.
%An interesting coincidence result holds even in infinite dimension.%sequence of Proposition \ref{p37} is the following result.

\begin{proposition}\label{p38} Let $X$ be a   Hausdorff locally convex topological vector space. Let $f:X\to\oR$ be a proper, convex and lower semicontinuous function. Let $U\subseteq \dom f$ be a self-segment-dense set. Then, $f=\overline{f_U}$ on $\co(U)$ and $\overline{f_U}=\overline{f_{\co(U)}}$ on $X.$
\end{proposition}
\begin{proof} Obviously $\co(U)\subseteq\dom\overline{f_U}.$ Let $x_0\in \co(U).$ Then $x_0=\sum_{i=1}^n\l_iu_i$ for some $u_1,...,u_n\in U$ and $\l_1,...,\l_n\in[0,1]$ with $\sum_{i=1}^n\l_i=1.$ Hence, $x_0\in\co\{u_1,...,u_n\}.$ Let $Y=\lin\{u_1,...,u_n\}$ and consider the function
$$\tilde{f}:Y\To\oR,\,\tilde{f}(x)=\left\{\begin{array}{ll} f(x),\mbox{ if }x\in\co\{u_1,...,u_n\}\\+\infty,\mbox{ otherwise.} \end{array}\right.$$
Then, $\tilde{f}$ is proper convex and lower semicontinuous, and in virtue of Theorem \ref{t2.8}, $\tilde{f}(u)=f(u)=\overline{f_U}(u)$  for all $u\in\co\{u_1,...,u_n\}\cap U.$
Obviously $Y$ is finite dimensional and according  to Lemma \ref{l25}, $\co\{u_1,...,u_n\}\cap U$ is dense, (actually is self-segment-dense), in $\co\{u_1,...,u_n\},$ hence according to Corollary 3.7 \cite{BD}, $\tilde{f}=\overline{f_U}$ on $\co\{u_1,...,u_n\}.$ In particular $f(x_0)=\overline{f_U}(x_0).$ Since $x_0\in\co(U)$ was arbitrary chosen, it follows that $f=\overline{f_U}$ on $\co(U).$

Let us denote $g=\overline{f_U}.$ According to Theorem \ref{t2.8} (a), $U$ is graphically dense in $\dom g$, hence $\co(U)$ is also graphically dense in $\dom g.$ By Theorem \ref{t2.8} (b), one has $\overline{g_{\co(U)}}=g$ on $X.$ But $g_{\co(U)}=f_{\co(U)}$ and the conclusion follows.
\end{proof}

\begin{remark}\rm According to Proposition \ref{p38}, if two proper convex and lower semicontinuous functions $f$ and $g$ coincide on a self-segment-dense set $U$ of their domain, then they also coincide on $\co(U).$ Indeed, in this case $\overline{f_U}=\overline{g_U}$, hence $f=\overline{f_U}=\overline{g_U}=g$ on $\co(U).$
\end{remark}

\begin{remark}\rm\label{r311} Note that under the hypothesis of Proposition \ref{p38}, one has\\ $\inf_{x\in U}f(x)=\inf_{x\in \co( U)}f(x).$ Indeed, according to Proposition \ref{p38} one has
$f=\overline{f_U}$ on $\co(U).$ On the other hand, according to Theorem \ref{t2.8} one has $\inf_{x\in U}f(x)=\inf_{x\in X}\overline{f_U}(x).$
Therefore, the following  inequalities $$\inf_{x\in U}f(x)\ge \inf_{x\in \co(U)}f(x)=\inf_{x\in \co(U)}\overline{f_U}(x)\ge \inf_{x\in X}\overline{f_U}(x)=\inf_{x\in U}f(x)$$
must hold with equality everywhere.

Moreover, if $\co(U)$ is closed then $f=\overline{f_U}$ on $X$, because $$U\subseteq \co(U)\subseteq\dom \overline{f_U}\subseteq \dom f\subseteq \cl(U)\subseteq\overline{\co}(U)=\co(U).$$ Hence, in this case,  $$\inf_{x\in U}f(x)=\inf_{x\in X}f(x).\,\,\,\,\,\,(*)$$
\end{remark}

The next example shows that $(*)$ fails even if $\co(U)$ is compact,  when we only assume  that $U$ is dense in $\dom f.$ It also shows, that the continuity assumption on $f$  in Corollary \ref{c1}, the segment-dense property of $U$ in Corolary \ref{c3} and the self-segment-denseness of $U$ in Proposition \ref{p38} are essential.
 \begin{example}\rm\label{exp39} Let  $X$ be an infinite dimensional real Hilbert space. Let $K=\left\{ x\in X:\left\Vert x\right\Vert \le 1\right\}$ be the unit ball of $X$ and let $U=\left\{ x\in X:\left\Vert x\right\Vert =1\right\}$. Then according to Example \ref{ex1}, $U$ is dense in $K$ with respect to the weak topology of $X,$ but is neither segment-dense not self-segment-dense in $K.$ Obviously $\co(U)=K$, which according to Banach-Alaoglu Theorem \cite{fabian} is weakly compact. Consider the function $$f:X\To\R,\,f(x)=\left\{\begin{array}{ll}\|x\|,\mbox{ if }x\in K\\+\infty,\mbox{ otherwise. }\end{array}\right.$$ Then, obviously $f$ is proper, convex and weakly lower semicontinuous and $\dom f=K.$ Nevertheless
$$\inf_{x\in U}f(x)=1$$
and
$$\min_{x\in X}f(x)=0.$$

It is also obvious the fact, that in this case $f\neq\overline{f_U}$ on $\co(U)$, since  $$\overline{f_U}(x)=\left\{\begin{array}{ll} 1,\mbox{ if }x\in K\\+\infty,\mbox{ otherwise. }\end{array}\right.$$
\end{example}

Next we present some other interesting conditions that ensure the validity of $(*)$.

\begin{proposition}\label{p312}  Let $X$ be a   Hausdorff locally convex topological vector space. Let $f:X\To \oR$ be a proper, convex and lower semicontinuous function, and let $U\subseteq \dom f$ be self-segment-dense in $\dom f.$ Assume further that $\co(U)$ is segment-dense in $\dom f.$
Then, $\overline{f_U}=\overline{f_{\co(U)}}=f$ on $X$ and $\inf_{x\in U}f(x)=\inf_{x\in X}f(x).$
\end{proposition}
\begin{proof}
 According to Proposition \ref{p38},  $\overline{f_U}=\overline{f_{\co(U)}}$ on $X$.
According to Corollary \ref{c3}, $\overline{f_{\co(U)}}=f$ on $X.$ Hence, $\overline{f_U}=\overline{f_{\co(U)}}=f$ on $X$.
From Theorem \ref{t2.8} (a), one has $\inf_{x\in U}f(x)=\inf_{x\in X}\overline{f_U}(x),$ consequently $\inf_{x\in U}f(x)=\inf_{x\in X}f(x).$
\end{proof}

\begin{remark}\rm\label{r313} Note that each of the following  conditions  ensure that $\co(U)$ is segment-dense in $\dom f.$
\begin{itemize}
\item[(a)] $\mbox{For all }x\in \dom f\mbox{ there exists }y\in\co(U),\mbox{ such that }[y,x)\subseteq\co(U).$
\item[(b)] The interior of $\co(U)$ is non-empty. In this case it is known that the following relation, sometimes called line  segment principle \cite{M,RW}, holds
\begin{equation}\label{e6}
\l\inte( \co (U)) +(1-\l)\overline{\co}(U) \subseteq\inte( \co( U)) ,\,\forall\l\in(0,1].
\end{equation}
\item[(c)] $X$ is finite dimensional. In this case $\co (U)$ has nonempty relative interior, and a similar relation to (\ref{e6}) holds.
\end{itemize}
\end{remark}
\begin{remark}\rm\label{r312} If $X$ is a Banach space then, according to Mazur Theorem \cite{fabian}, the weak and strong closure of a convex set coincides.
Therefore, if we endow $X$ with the weak topology, in  (\ref{e6}) it is enough to assume  that the strong interior of $\co (U)$ is nonempty.
\end{remark}
\begin{remark}\rm Example \ref{exp39} shows that also in hypothesis of Proposition \ref{p312} the self-segment-dense assumption on $U$ cannot be replaced by the usual denseness assumption. Indeed, for the sets  $X, K, U$ and the function $f$ considered  Example \ref{exp39}  all the assumptions in the hypothesis of Proposition \ref{p312} are fulfilled excepting $U$ is self-segment-dense. Here $U$ is only dense. Since $\co(U)=K=\dom f,$ obviously $\co(U)$ is segment-dense in $\dom f.$ Nevertheless $$1=\inf_{x\in U}f(x)>0=\min_{x\in X}f(x).$$
\end{remark}

\section{Minimax results on dense sets}

In this section we prove several minimax theorems on dense sets. We obtain  some results  where the conditions imposed to the bifunction that describes the minimax problem are considered relative to a dense set. Several examples and counterexamples circumscribe the results of this section. In particular we show that the general minimax results of Fan and Sion cannot be extended to usual dense sets.

Based on the results of Section 3, we are able to give a proof for the following minimax theorem, an extension of Fan's minimax result on self-segment-dense sets.

\begin{theorem}\label{tm2} Let $K$ be a  nonempty, compact and convex subset of the Hausdorff locally convex topological vector space $X$ and let $Y$ be an arbitrary  nonempty set. Let $U\subseteq K$ be a self-segment-dense set in $K$ and assume that $\co(U)$ is segment-dense in $K.$ Consider further the mapping $f:K\times Y\To\R,$ and assume  that the following assumptions are fulfilled.
\begin{itemize}
\item[(i)] The map $x\To f(x,y)$ is proper, convex and lower semicontinuous on $K$ for all $y\in Y$.
\item[(ii)] The map $y\To f(x,y)$ is concavelike, for all $x\in  K.$
\end{itemize}
Then,
$$\inf_{x\in  U}\sup_{y\in Y}f(x,y)=\sup_{y\in Y}\inf_{x\in U}f(x,y).$$
\end{theorem}
\begin{proof} In virtue of Proposition \ref{p312},  one has $$\sup_{y\in Y}\inf_{x\in U}f(x,y)= \sup_{y\in Y}\min_{x\in K}f(x,y).$$ Theorem \ref{fan} assures that $$\sup_{y\in Y}\min_{x\in K}f(x,y)=\min_{x\in K}\sup_{y\in Y}f(x,y).$$
On  the other hand, the function $g:K\To\oR,\,g(x)=\sup_{y\in Y}f(x,y)$ is proper, convex and lower semicontinuous as a pointwise supremum of a family of proper, convex and lower semicontinuous functions, hence by Proposition \ref{p312} one obtains
$$\min_{x\in K}\sup_{y\in Y}f(x,y)=\inf_{x\in U}\sup_{y\in Y}f(x,y).$$

Thus, we have
$$ \sup_{y\in Y}\inf_{x\in U}f(x,y)= \sup_{y\in Y}\min_{x\in K}f(x,y)=\min_{x\in K}\sup_{y\in Y}f(x,y)=\inf_{x\in U}\sup_{y\in Y}f(x,y),$$ and the conclusion follows.
\end{proof}
\begin{remark} Note that, in case $Y$ is a convex subset of a topological vector space, then  (ii) can be replaced by the following:
the map $y\To f(x,y)$ is upper semicontinuous and quasiconcave, for all $x\in  K.$ Then, the conclusion of Theorem \ref{tm2} follows by using Theorem \ref{sion} instead of Theorem \ref{fan} in its proof. The same argument is valid in every result presented bellow.
\end{remark}

\begin{remark}\rm
 For every fixed $y\in Y$, we denote by $\overline{f_{U}}(\cdot,y)$, the function defined as $\epi\overline{f_{U}}(\cdot,y)=\cl(\epi_{U}f(\cdot,y)).$ If one obtains a condition that assures the domain of $\overline{f_U}(\cdot,y)$ being the same closed subset of the compact set $K$ for every $y\in Y$, in particular $\dom\overline{f_U}(\cdot,y)=K$ for every $y\in Y$, then one  can assume that the conditions $(i)$ and $(ii)$  in Theorem \ref{tm2} are fulfilled only on $U.$  More precisely the following result holds.%renounce to the restrictive assumption of compactness of $\co(U).$
 \end{remark}
 \begin{theorem}\label{tm3}
 Let $K$ be a  nonempty, compact and convex subset of the Hausdorff locally convex topological vector space $X$ and let $Y$ be an arbitrary  nonempty set. Let $U\subseteq K$ be a self-segment-dense set in $K$ and assume that $\co( U)$ is segment-dense in $K.$ Consider further the mapping $f:K\times Y\To\R,$ and assume  that the following assumptions are fulfilled.
\begin{itemize}
\item[(i)] The map $x\To f(x,y)$ is proper, convex and lower semicontinuous on $U$ for all $y\in Y$.
\item[(ii)] The map $y\To f(x,y)$ is concavelike, for all $x\in  U.$
\item[(iii)] For every $y\in Y$, $\dom \overline{f_U}(\cdot,y)=K.$
\end{itemize}
Then,
$$\inf_{x\in  U}\sup_{y\in Y}f(x,y)=\sup_{y\in Y}\inf_{x\in U}f(x,y).$$
\end{theorem}
\begin{proof} In virtue of Theorem \ref{t2.8}, one has $$\sup_{y\in Y}\inf_{x\in U}f(x,y)= \sup_{y\in Y}\min_{x\in K}\overline{f_U}(x,y).$$ According to Theorem \ref{t2.8}, the function $x\To \overline{f_U}(x,y)$ is convex, proper and lower semicontinuous for all $y\in Y.$

We show that $\overline{f_{U}}$ is concavelike in its second variable, that is, for every $y_1,y_2\in Y$ and $t\in[0,1]$ there exists $y_3\in Y$ such that $$(1-t)\overline{f_{U}}(x,y_1)+t\overline{f_{U}}(x,y_2)\le \overline{f_{U}}(x,y_3)\mbox{ for all }x\in\co(U).$$ From the hypothesis of the theorem, we have that for every $y_1,y_2\in Y$ and $t\in[0,1]$ there exists $y_3\in Y$ such that $(1-t){f}(u,y_1)+t {f}(u,y_2)\le {f}(u,y_3)$ for all $u\in U.$ From Theorem \ref{t2.8} we get that $f(u,y_j)=\overline{f_{U}}(u,y_j)$ for all $u\in U,\, j=1,2,3.$ Thus, $(1-t)\overline{f_{U}}(u,y_1)+t \overline{f_{U}}(u,y_2)\le \overline{f_{U}}(u,y_3)$ for all $u\in U.$ Let $x\in \co(U).$ By the construction of $\overline{f_{U}}(\cdot,y_3)$ we obtain that there exists a net $((u^i,r_i))\subseteq\epi f_U(\cdot,y_3),$ such that $(u^i,r_i)\To (x,\overline{f_{U}}(x,y_3)).$ We have $$r_i\ge f(u^i,y_3)=\overline{f_{U}}(u^i,y_3),\mbox{ and }r_i\To \overline{f_{U}}(x,y_3),$$ consequently
$$\overline{f_{U}}(x,y_3)=\liminf r_i\ge\liminf\overline{f_{U}}(u^i,y_3).$$
 Since $\overline{f_{U}}(\cdot,y_j),\,j=1,2,3$  is lower semicontinuous we have, that
$$(1-t)\overline{f_{U}}(x,y_1)+t \overline{f_{U}}(x,y_2)\le$$
$$\le(1-t)\liminf\overline{f_{U}}(u^i,y_1)+t \liminf\overline{f_{U}}(u^i,y_2)\le$$
$$\le\liminf((1-t)\overline{f_{U}}(u^i,y_1)+t\overline{f_{U}}(u^i,y_2))\le$$
$$\le\liminf \overline{f_{U}}(u^i,y_3)=\overline{f_{U}}(x,y_3).$$
 Hence, $$(1-t)\overline{f_{U}}(x,y_1)+t \overline{f_{U}}(x,y_2)\le \overline{f_{U}}(x,y_3)\mbox{ for all }x\in\co(U).$$

Now, Theorem \ref{fan} assures that $$\sup_{y\in Y}\min_{x\in K}\overline{f_U}(x,y)=\min_{x\in K}\sup_{y\in Y}\overline{f_U}(x,y).$$
On  the other hand, the function $g:K\To\oR,\,g(x)=\sup_{y\in Y}\overline{f_U}(x,y)$ is proper, convex and lower semicontinuous as a pointwise supremum of a family of proper, convex and lower semicontinuous functions, hence according to Proposition \ref{p312} one has
$$\min_{x\in K}\sup_{y\in Y}\overline{f_U}(x,y)=\inf_{x\in U}\sup_{y\in Y}\overline{f_U}(x,y).$$

But, according to Theorem \ref{t2.8}, we have $f(x,y)=\overline{f_U}(x,y)$ for all $x\in U$.
Hence,
$$ \inf_{x\in U}\sup_{y\in Y}\overline{f_U}(x,y)=\inf_{x\in U}\sup_{y\in Y}f(x,y),$$ and the conclusion follows.
\end{proof}

First of all we would like to emphasize that the conclusions of Theorem \ref{tm2} and Theorem \ref{tm3} does not remain valid if in its hypothesis we assume only that the set $U$ is dense in $K$ as the next example shows.  Moreover, the assumptions imposed on the bifunction $f$ in Theorem \ref{fan} and Theorem \ref{sion}  are also satisfied, however their conclusions fail. This fact shows that the general results of Fan and Sion cannot be extended on general dense sets.

\begin{example}\rm\label{extm} Let  $X$ be an infinite dimensional real Hilbert space. Let $K=Y=\left\{ x\in X:\left\Vert x\right\Vert \le 1\right\}$ be the unit ball of $X$ and let $U=\left\{ x\in X:\left\Vert x\right\Vert =1\right\}$. Then according to Example \ref{ex1}, $U$ is dense in $K$ with respect to the weak topology of $X,$ but is not self-segment-dense in $K.$ Obviously $\co(U)=K$ which according to Banach-Alaoglu Theorem is weakly compact. Obviously, in this case $\co(U)$ is segment-dense in $K.$  Consider the function $$f:K\times Y\To\R,\,f(x,y)=\<x,y\>.$$ Then, it can easily be verified that the conditions in the hypotheses  of Theorem \ref{fan}, Theorem \ref{sion}, Theorem \ref{tm2} and Theorem \ref{tm3} are fulfilled. Nevertheless
$$\inf_{x\in U}\sup_{y\in Y}f(x,y)=1$$
and
$$\sup_{y\in Y}\inf_{x\in U}f(x,y)=0.$$
\end{example}

As immediate consequences of Theorem \ref{tm3} we have the following results.

\begin{corollary}\label{tm} Let $K$ be a  nonempty, compact and convex  subset of the Hausdorff locally convex topological vector space $X$ and let $Y$ be an arbitrary nonempty set. Let $U\subseteq K$ be a self-segment-dense set. Consider the mapping $f:K\times Y\To\R,$ and assume  that the following assumptions are fulfilled.
\begin{itemize}
\item[(i)] The map $x\To f(x,y)$ is  convex  and lower semicontinuous on $U$ for all $y\in Y$.
\item[(ii)] The map $y\To f(x,y)$ is concavelike, for all $x\in  U.$
\item[(iii)] For every $y\in Y$, $\dom \overline{f_U}(\cdot,y)=K.$
\end{itemize}
Assume further that one of the following conditions hold.
\begin{itemize}
\item[(a)] $\co(U)=K.$
\item[(b)] $\mbox{For all }x\in K \mbox{ there exists }y\in\co(U),\mbox{ such that }[y,x)\subseteq\co(U).$
\item[(c)] The interior of $\co(U)$ is non-empty.
\item[(d)] $X$ is finite dimensional.
\end{itemize}
Then,
$$\inf_{x\in  U}\sup_{y\in Y}f(x,y)=\sup_{y\in Y}\inf_{x\in U}f(x,y).$$
\end{corollary}
\begin{proof} Observe that according to Remark \ref{r313}, each of the conditions (a)-(d) assures that $\co(U)$ is segment-dense in $K.$ The conclusion follows from Theorem \ref{tm3}.
\end{proof}
\begin{corollary}\label{c4} Let $K$ be a  nonempty and convex subset of the Hausdorff locally convex space $X$ and let $Y$ be an arbitrary nonempty set. Let $U\subseteq K$ be a self-segment-dense set, let $\mathcal{S}\subseteq U$ be a  subset of $U$ and assume  that $\co(\mathcal{S})$ is compact.  Consider further the mapping $f:K\times Y\To\R,$ and assume    that the following assumptions are fulfilled.
\begin{itemize}
\item[(i)] The map $x\To f(x,y)$ is  convex  and lower semicontinuous  on $\co(\mathcal{S})\cap U$ for all $y\in Y$.
\item[(ii)] The map $y\To f(x,y)$ is concavelike, for all $x\in \co(\mathcal{S})\cap U.$
\end{itemize}
Then,
$$\inf_{x\in \co(\mathcal{S})\cap U}\sup_{y\in Y}f(x,y)=\sup_{y\in Y}\inf_{x\in\co(\mathcal{S})\cap U}f(x,y).$$
\end{corollary}
\begin{proof}  Note that according to Lemma \ref{l26}, $\co(\mathcal{S})\cap U$ is self-segment-dense in $\co(\mathcal{S}).$ Obviously $\co(\co(\mathcal{S})\cap U)=\co(\mathcal{S}),$ which shows that $\co(\co(\mathcal{S})\cap U)$ is segment-dense in $\co(\mathcal{S})$. Hence, Theorem \ref{tm3} can be applied for the function
$$\tilde{f}:\co(S)\times Y\To\R,\,\tilde{f}(x,y)=f(x,y).$$
 \end{proof}
 We would like to  emphasize that the conclusion of Corollary \ref{c4} fails even in finite dimension if in its hypothesis we replace the condition $U$ is self-segment-dense in $K$ by the condition  that $U$ is dense, or segment-dense  in $K.$
\begin{example}\rm\label{ex32} Let $K=Y=\{(x^1,x^2)\in\R^2:(x^1)^2+(x^2)^2\le 1\}$ be the closed unit ball of $\R^2$ and let $U=K\setminus \left[\left(-\frac12,0\right),\left(\frac12,0\right)\right].$ Obviously $U$ is dense in $K$, and also segment-dense in the sense of The Luc, but not self-segment-dense, since for $u_1=(-1,0),u_2=(1,0)\in U$ one has $$\cl([u_1,u_2]\cap U)=\left[\left(-1,0\right),\left(-\frac12,0\right)\right[\cup \left]\left(\frac12,0\right),\left(1,0\right)\right]\neq[u_1,u_2].$$
Consider the bifunction $f:K\times K\To\R,\,f(x,y)=f((x^1,x^2),(y^1,y^2))=x^1y^1+x^2y^2.$ Then, it is straightforward that for every subset $\mathcal{S}\subseteq U$ the conditions (i) and (ii) of Corollary \ref{c4} are satisfied. Nevertheless its conclusion fails as we will show in what follows. This is due to the fact that $U$ is not self-segment-dense in $K.$

Indeed, let $\mathcal{S}=\{u_1,u_2\}.$ Then, $$\inf_{x\in \co(\mathcal{S})\cap U}\sup_{y\in Y}f(x,y)=\inf_{x^1\in [-1,-\frac12[\cup]\frac12,1]}\sup_{y\in K}x^1y^1=\inf_{x^1\in [-1,-\frac12[\cup]\frac12,1]}|x^1|=\frac12.$$
On the other hand,
$$\sup_{y\in Y}\inf_{x\in \co(\mathcal{S})\cap U}f(x,y)=\sup_{y\in K}\inf_{x^1\in [-1,-\frac12[\cup]\frac12,1]}x^1y^1=\sup_{y\in K}-|y^1|=0.$$
\end{example}

In the next Corollary we assume that $Y$ is finite.

\begin{corollary}\label{ctm1} Let $K$ be a  nonempty and convex  subset of the Hausdorff locally convex topological vector space $X$ and let $Y$ be an arbitrary  nonempty and finite set. Let $U\subseteq K$ be a self-segment-dense set in $K.$ Consider further the mapping $f:K\times Y\To\R,$ and assume    that the following assumptions are fulfilled.
\begin{itemize}
\item[(i)] The map $x\To f(x,y)$ is  convex  and lower semicontinuous on $U$ for all $y\in Y$.
\item[(ii)] The map $y\To f(x,y)$ is concavelike, for all $x\in  U.$
\item[(iii)] For every $y\in Y$, $\inf_{x\in U} f(x,y)$ is attained.
\end{itemize}
Then,
$$\inf_{x\in  U}\sup_{y\in Y}f(x,y)=\sup_{y\in Y}\min_{x\in U}f(x,y).$$
\end{corollary}
\begin{proof} The result is a consequence of Theorem \ref{tm3}.  Let $S=\cup_{y\in Y} \{u_y\},$ where $u_y\in S_y=\{u\in U: \min_{x\in U}f(x,y)=f(u,y)\}.$ Then ${\co}(S)\subseteq \dom f(\cdot,y)=K$ for all $y\in Y.$ Obviously ${\co}(S)$ is compact as a convex hull of a finite set, further according to Lemma \ref{l25}, $\co(S)\cap U$ is self-segment-dense in ${\co}(S).$ Let $\tilde{f}$ be the restriction of $f$ on ${\co} (S)\times Y,$ that is
 $$\tilde{f}:{\co}(S)\times Y\To \R,\, \tilde{f}(x,y)=f(x,y).$$ Then, Theorem \ref{tm3} can be applied, hence
 $$\inf_{x\in \co(S)\cap U}\sup_{y\in Y}\tilde{f}(x,y)=\sup_{y\in Y}\inf_{x\in\co(S)\cap U}\tilde{f}(x,y).$$
 In other words
 $$\inf_{x\in \co(S)\cap U}\sup_{y\in Y}{f}(x,y)=\sup_{y\in Y}\inf_{x\in\co(S)\cap U}{f}(x,y).$$
But, by construction of $S$ one has that for every $y\in Y,$ $\min_{x\in U} f(x,y)=\inf_{x\in\co(S)\cap U}{f}(x,y)$, hence
 $$\sup_{y\in Y}\min_{x\in  U}{f}(x,y)=\sup_{y\in Y}\inf_{x\in\co(S)\cap U}{f}(x,y)=\inf_{x\in \co(S)\cap U}\sup_{y\in Y}{f}(x,y)\ge\inf_{x\in  U}\sup_{y\in Y}{f}(x,y).$$

 Since $\inf_{x\in  U}\sup_{y\in Y}{f}(x,y)\ge \sup_{y\in Y}\inf_{x\in  U}{f}(x,y)$ always holds the conclusion follows.
\end{proof}

Observe that  in the previous result  we assumed that $Y$ is finite. This fact assured that the set $\co (S)$ is compact. Note that Theorem \ref{tm3} is a great theoretical result, nevertheless the segment-dense requirement of $\co (U)$ in some cases might be restrictive. In what follows we present a minimax result on general dense sets, where we do not assume that $\co (U)$ is segment-dense. However, we have to consider some quite strong conditions imposed  to the bifunction that describes the minimax problem.
 Further, observe that if we replace the lower semicontinuity assumption of the maps $x\To f(x,y)$ on $U$ for all $y\in Y$ in the hypothesis of Theorem \ref{tm3} by their continuity assumption on $K$, then Corollary \ref{c1}  assures that $\inf_{x\in U}f(x,y)=\inf_{x\in K}f(x,y)$ for all $y\in Y.$ Therefore, one can renounce to the assumption $U$ is self-segment-dense and  $\co(U)$ is segment-dense in $K$, in the hypothesis of Theorem \ref{tm3}, provided the condition $(i)$ is replaced by the condition: the map $x\To f(x,y)$ is  convex and continuous on $K$ for all $y\in Y$. In this case due to the continuity of $f$ in its first variable, the condition $(ii)$ becomes: the map $y\To f(x,y)$ is concavelike, for all $x\in  K.$ Further, we will need an extra assumption in order to assure the continuity of the function $g(x)=\sup_{y\in Y}f(x,y).$

 For $K\subseteq X$  let us denote by $C(K)$ the space of continuous real valued functions, that is
$C(K)=\{f:K\to\R:f\mbox{ continuous }\}.$
A subset $S\subseteq C(K)$ is said to be equicontinuous if for every $x \in K$ and every $\e> 0$, $x$ has a neighborhood $U_x$ such that
$\forall y \in U_x\cap K, \forall f \in S, |f(y) - f(x)| < \e.$
A set $S\subseteq C(K)$ is said to be pointwise bounded if for every $x \in K$, $\sup_{f \in S} | f(x) |  < \infty.$
If $K$ is also compact, we can endow $C(K)$ with the uniform norm, $\|f\|=\sup_{x\in K}|f(x)|.$
In case $K$ is compact, the Arzel\`a-Ascoli theorem \cite{DS}, affirms that  a subset $S$ of $C(K)$ is relatively compact in the topology induced by the uniform norm of $C(K),$ if and only if it is equicontinuous and pointwise bounded. These concepts allow us to obtain a result in which we do not assume the self-segment-denseness of $U$ or the segment-denseness of $\co (U).$

\begin{theorem}\label{tm1} Let $K$ be a  nonempty, compact and convex subset of the Hausdorff locally convex topological vector space $X$ and let $Y$ be an arbitrary  nonempty set. Let $U\subseteq K$ be a dense set in $K.$ Consider further the mapping $f:K\times Y\To\R,$ and assume  that the following assumptions are fulfilled.
\begin{itemize}
\item[(i)] The map $x\To f(x,y)$ is  convex  on $K$ for all $y\in Y$.
\item[(ii)] The map $y\To f(x,y)$ is concavelike, for all $x\in  K.$
%\item[(iii)] For every $y\in Y$, $\inf_{x\in U} f(x,y)$ is attained.
\item[(iii)] The family $(f(\cdot,y))_{y\in Y}$ is an  equicontinuous  family of $C(K)$.
\item[(iv)] For all $x\in K$ one has $\sup_{y\in Y}f(x,y)< \infty.$
\end{itemize}
Then,
$$\inf_{x\in  U}\sup_{y\in Y}f(x,y)=\sup_{y\in Y}\inf_{x\in U}f(x,y).$$
\end{theorem}
\begin{proof} Note that condition $(iii)$ ensures that $x\To f(x,y)$ is  continuous  on $K$ for all $y\in Y$. In virtue of Corollary \ref{c1}, one has \begin{equation}\label{e8}\sup_{y\in Y}\inf_{x\in U}f(x,y)= \sup_{y\in Y}\min_{x\in K}f(x,y).\end{equation} Theorem \ref{fan} assures that \begin{equation}\label{e9}\sup_{y\in Y}\min_{x\in K}f(x,y)=\min_{x\in K}\sup_{y\in Y}f(x,y).\end{equation}
On  the other hand, the function $g:K\To\oR,\,g(x)=\sup_{y\in Y}f(x,y)$ is proper, convex and lower semicontinuous as a pointwise supremum of a family of proper, convex and lower semicontinuous functions. From (iv) we have that  $g$ is not extended valued, that is $g(x)\in\R$ for all $x\in K.$ We show that $g$ is continuous on $K.$ Indeed, let $x_0\in K.$ Since the family $(f(\cdot,y))_{y\in Y}$ is equicontinuous one has that for every $\e>0$ there exists $U_0$ a neighbourhood of $x_0$, such  that, for all $x\in U_0\cap K$ and for all $y\in Y$ one has \begin{equation}\label{e7}|f(x,y)-f(x_0,y)|<\frac\e2.\end{equation}
Let us fix $\e>0$ and let $x_1\in U_0\cap K.$ Then by the definition of supremum, there exist $y_0,y_1\in Y$ such that $f(x_0,y_0)+\frac\e2>g(x_0)$ and $f(x_1,y_1)+\frac\e2>g(x_1).$ Obviously, $g(x_0)\ge f(x_0,y_1)$ and $g(x_1)\ge f(x_1,y_0).$ Hence,
$$f(x_1,y_0)-f(x_0,y_0)-\frac\e2\le g(x_1)-g(x_0)\le f(x_1,y_1)-f(x_0,y_1)+\frac\e2,$$
and using (\ref{e7}), one gets
$$|g(x_1)-g(x_0)|\le \e.$$
According to Corollary \ref{c1}, one has $\inf_{x\in U}g(x)=\min_{x\in K}g(x).$ In other words, \begin{equation}\label{e10}\min_{x\in K}\sup_{y\in Y}f(x,y)=\inf_{x\in U}\sup_{y\in Y}f(x,y).\end{equation}
The conclusion of theorem follows from (\ref{e8}), (\ref{e9}) and (\ref{e10}).
\end{proof}

\begin{remark}\rm\label{r41}
Example \ref{extm} becomes   again a counterexample for Theorem \ref{tm1}. However, note that  in this case the contradiction is not provided by the fact that the set $U$ considered is not self-segment-dense, but by the fact that the corresponding family of functions, $(f(\cdot,y))_{y\in Y}=(\<\cdot,y\>)_{y\in B}$ is not an  equicontinuous family in the weak topology of $X.$
\end{remark}

\section{Dense families of functionals}

In this section we apply our minimax results in order to prove the denseness of some family of functionals in $B (Y),$ where $ B (Y)$ is the space of bounded functions on $Y$ endowed with the topology of uniform norm. %and $C (Y).$
We  show that also here the concept of a self-segment-dense set is essential in order to obtain these results.

\begin{remark}\rm Note that in case $\{f(\cdot,y)\}_{y\in Y}$ is an equicontinuous family of $C(K)$, then for any dense set $U\subseteq K$ one has that
$$\forall k\in K,\,\forall \e>0,\,\exists u\in U\mbox{ such that }\sup_{y\in Y}|f(k,y)-f(u,y)|<\e.$$ In other words, $\|f(u,\cdot)-f(k,\cdot)\|< \e$ in the uniform norm of $B (Y).$ Consequently, $\{f(u,\cdot)\}_{u\in  U}$  is dense in $\{f(x,\cdot)\}_{x\in K}.$
\end{remark}

The next abstract theorem provides the denseness of the set of functions $\{f(u,\cdot)\}_{u\in U}$ in $\{f(x,\cdot)\}_{x\in K}$ in $B (Y)$, without the assumption of equicontinuity of the family $\{f(\cdot,y)\}_{y\in Y}$.

\begin{theorem}\label{ta0} Let $K$ be a  nonempty  subset of the Hausdorff locally convex topological vector space $X$ and let $Y$ be a nonempty subset of a locally convex topological vector space. Let $U\subseteq K$ be a dense set in $K.$ Consider further the bounded mapping $f:K\times Y\To\R,$ and assume  that the following assumptions are fulfilled.
\begin{itemize}
\item[(i)] The map $x\To f(x,y)$ is   continuous on $K$ for all $y\in Y$.
\item[(ii)] The map $y\To f(x,y)$ is  continuous, for all $x\in  K.$
\item[(iii)] For every $k\in K$, $\sup_{y\in Y}\inf_{x\in U} |f(x,y)-f(k,y)|=\inf_{x\in U}\sup_{y\in Y} |f(x,y)-f(k,y)|.$
%\item[(iii)] The family $(f(\cdot,y))_{y\in Y}$ is an  equicontinuous and pointwise bounded family of $C(K)$.
%\item[(iv)] $\inf_{x\in  U}\sup_{y\in Y_0}f(x,y)=\sup_{y\in Y_0}\inf_{x\in U}f(x,y),$ for every closed  subset $Y_0\subseteq Y.$
\end{itemize}
Then, $\{f(u,\cdot)\}_{u\in  U}$  is dense in $\{f(x,\cdot)\}_{x\in K}\subseteq B (Y).$
\end{theorem}
\begin{proof} Let $k\in K.$ We show that for all $\e>0$ there exists $u\in  U$ such that
$$\sup_{y\in Y}|f(u,y)-f(k,y)|< \e.$$
In other words, $\|f(u,\cdot)-f(k,\cdot)\|< \e$ in the uniform norm of $B (Y).$ This means that $\{f(u,\cdot)\}_{u\in  U}$  is dense in $\{f(x,\cdot)\}_{x\in K}.$

Let $\e>0$ and $y_0\in Y$ be fixed. Since the mapping $x\To f(x,y_0)$ is continuous there exists an open neighbourhood $V$ of $k$ such that for all $x\in V$ one has $$|f(x,y_0)-f(k,y_0)|<\frac{\e}{8}.$$
Since $U$ is dense in $K$, there exists $u_0\in V\cap U$ such that
\begin{equation}\label{e01}
|f(u_0,y_0)-f(k,y_0)|<\frac{\e}{8}.
\end{equation}
On the other hand the mapping $y\To f(u_0,y)$ is continuous at $y_0$, hence there exists $V_0$ an open convex neighbourhood of $y_0$ such that for all $y\in V_0$ one has
\begin{equation}\label{e02}
|f(u_0,y)-f(u_0,y_0)|<\frac{\e}{8}.
\end{equation}
It can easily be observed, that $(\ref{e01})$ and $(\ref{e02})$ lead to
\begin{equation}\label{e03}
|f(u_0,y)-f(k,y_0)|<\frac{\e}{4},\,\forall y\in V_0.
\end{equation}
Finally, the mapping $y\To f(k,y)$ is continuous at $y_0$, hence there exists $V'_0$ an open convex neighbourhood of $y_0$ such that for all $y\in V'_0$ one has
\begin{equation}\label{e04}
|f(k,y)-f(k,y_0)|<\frac{\e}{4}.
\end{equation}

From $(\ref{e03})$ and $(\ref{e04})$ we obtain, that for all $y\in V(y_0)=V_0\cap V'_0$ one has
\begin{equation}\label{e05}
|f(u_0,y)-f(k,y)|<\frac{\e}{2}.
\end{equation}
In other words, for every $y_0\in Y$ there exist $V(y_0)$, an open and convex neighbourhood of $y_0,$ and $u_0\in U$ such that
$$\sup_{y\in V(y_0)}|f(u_0,y)-f(k,y)|\le\frac{\e}{2}.$$

Obviously $\cup_{y_0\in Y}V(y_0)$ is an open cover of the set $Y$.

Note that for every $y_0\in Y$ there exists $V(y_0)$ such that $y_0\in V(y_0),$  hence
$$\inf_{u\in U}|f(u,y_0)-f(k,y_0)|\le |f(u_0,y_0)-f(k,y_0)|\le\frac{\e}{2}.$$
Thus,
$$\sup_{y\in Y}\inf_{u\in U}|f(u,y)-f(k,y)|\le\frac{\e}{2}<\e.$$

The latter relation combined with (iii) leads to
$$\inf_{u\in U}\sup_{y\in Y}|f(u,y)-f(k,y)|<\e.$$ Consequently, there exists $u^*\in U$ such that
$$\sup_{y\in Y}|f(u^*,y)-f(k,y)|<\e.$$
\end{proof}

When $Y$ is compact the following result holds in $C (Y),$ where $C(Y)$ is the space of continuous functions on $Y$ endowed with the uniform norm.

\begin{theorem}\label{ta} Let $K$ be a  nonempty convex subset of the Hausdorff locally convex topological vector space $X$ and let $Y$ be a compact and symmetric subset of a locally convex topological vector space. Let $U\subseteq K$ be a self-segment-dense set in $K.$ Consider further the mapping $f:K\times Y\To\R,$ and assume  that the following assumptions are fulfilled.
\begin{itemize}
\item[(i)] The map $x\To f(x,y)$ is  convex on $U$ and continuous on $K$ for all $y\in Y$.
\item[(ii)] The map $y\To f(x,y)$ is affine and continuous, for all $x\in  K.$
%\item[(iii)] For every $y\in Y$, $\inf_{x\in U} f(x,y)$ is attained.
%\item[(iii)] The family $(f(\cdot,y))_{y\in Y}$ is an  equicontinuous and pointwise bounded family of $C(K)$.
%\item[(iv)] $\inf_{x\in  U}\sup_{y\in Y_0}f(x,y)=\sup_{y\in Y_0}\inf_{x\in U}f(x,y),$ for every closed  subset $Y_0\subseteq Y.$
\end{itemize}
Then, $\{f(u,\cdot)\}_{u\in  U}$  is dense in $\{f(x,\cdot)\}_{x\in K}\subseteq C (Y),$ where $ C (Y)$ is the space of continuous functions on $Y$ endowed with the topology of uniform norm.
\end{theorem}
\begin{proof} Let $k\in K$ and let $\e>0$. As in the proof of Theorem \ref{ta0} one can show that  for every $y_0\in Y$ there exist $V(y_0)$, an open and convex neighbourhood of $y_0,$ and $u_0\in U$ such that
$$\sup_{y\in V(y_0)}|f(u_0,y)-f(k,y)|\le\frac{\e}{2}.$$

Obviously $\cup_{y_0\in Y}V(y_0)$ is an open cover of the compact set $Y$, hence it contains a finite subcover. In other words, there exist $y_1,...,y_n\in Y$ and $u_1,...,u_n\in U$ such that $Y=\cup_{i=1}^n V(y_i)$ and
$$\sup_{y\in V(y_i)}|f(u_i,y)-f(k,y)|\le\frac{\e}{2}\mbox{ for all }i\in\{1,2,...,n\}.$$

Note that for every $y_0\in Y$ there exists $V(y_j)$ such that $y_0\in V(y_j).$ Further, it is obvious that
$$\inf_{u\in\{u_1,...,u_n\}}|f(u,y_0)-f(k,y_0)|\le |f(u_j,y_0)-f(k,y_0)|\le\frac{\e}{2}$$
hence,
$$\sup_{y\in Y}\inf_{u\in\{u_1,...,u_n\}}|f(u,y)-f(k,y)|\le\frac{\e}{2}.$$

The latter relation leads to
$$\sup_{y\in Y}\inf_{u\in\co\{u_1,...,u_n\}\cap U}(f(u,y)-f(k,y))\le\frac{\e}{2}.$$

Obviously $\co\{u_1,...,u_n\}$ is compact and according to Lemma \ref{l25}, the intersection $\co\{u_1,...,u_n\}\cap U$ is self-segment-dense in $\co\{u_1,...,u_n\}.$ We show that Corollary \ref{c1} can be applied to the function $g:K\times Y\To\R,\, g(x,y)=f(x,y)-f(k,y).$
Indeed, the mapping $x\To g(x,y)$ is convex on $U$  and continuous on $K$ for all $y\in Y$ hence it is also convex on $\co\{u_1,...,u_n\}\cap U,$ (since  $U$ is self-segment-dense in $K$), and lower semicontinuous on $\co\{u_1,...,u_n\}$ for all $y\in Y.$

We show that the map $y\To g(x,y)$ is concave (hence also concavelike) for all $x\in \co\{u_1,...,u_n\}\cap U.$

Indeed, let $x\in \co\{u_1,...,u_n\}\cap U.$ According to the hypothesis of the theorem, the mapping $y\To f(x,y)$ is affine, hence for every $y_1,y_2\in Y$ and $t\in[0,1]$ one has
$$g(x,(1-t)y_1+ty_2)=f(x,(1-t)y_1+ty_2)-f(k,(1-t)y_1+ty_2)=$$
$$(1-t)(f(x,y_1)-f(k,y_1))+t(f(x,y_2)-f(k,y_2))=(1-t)g(x,y_1)+tg(x,y_2).$$

By applying Corollary \ref{c4}, we obtain
$$\inf_{u\in\co\{u_1,...,u_n\}\cap U}\sup_{y\in Y}(f(u,y)-f(k,y))=\sup_{y\in Y}\inf_{u\in\co\{u_1,...,u_n\}\cap U}(f(u,y)-f(k,y))\le$$
$$\le\frac{\e}{2}<\e.$$
Hence, there exists $u^*\in\co\{u_1,...,u_n\}\cap U$ such that
$$\sup_{y\in Y}(f(u^*,y)-f(k,y))< \e.$$

Conversely, since $Y$ is symmetric and $f$ is affine in the second variable we have
$$\sup_{y\in Y}(-f(u^*,y)+f(k,y))=\sup_{y\in Y}(f(u^*,-y)-f(k,-y))$$
and $$\sup_{y\in Y}(f(u^*,-y)-f(k,-y))=\sup_{y\in Y}(f(u^*,y)-f(k,y))< \e.$$
Hence, $$\sup_{y\in Y}|f(u^*,y)-f(k,y)|< \e.$$
\end{proof}

In the next result we drop the compactness assumption on $Y$, but we assume instead some minimax results.

\begin{theorem}\label{ta1} Let $K$ be a  nonempty convex subset of the Hausdorff locally convex topological vector space $X$ and let $Y$ be a closed convex bounded and symmetric subset of a locally convex topological vector space. Let $U\subseteq K$ be a self-segment-dense set in $K.$ Consider further the bounded mapping $f:K\times Y\To\R,$ and assume  that the following assumptions are fulfilled.
\begin{itemize}
\item[(i)] The map $x\To f(x,y)$ is  convex on $U$ and continuous on $K$ for all $y\in Y$.
\item[(ii)] The map $y\To f(x,y)$ is affine and continuous on $Y$ for all $x\in  K.$
%\item[(iii)] For every $y\in Y$, $\inf_{x\in U} f(x,y)$ is attained.
%\item[(iii)] The family $(f(\cdot,y))_{y\in Y}$ is an  equicontinuous and pointwise bounded family of $C(K)$.
\item[(iii)] $\inf_{x\in  U}\sup_{y\in Y_0}f(x,y)=\sup_{y\in Y_0}\inf_{x\in U}f(x,y),$ for every closed convex subset $Y_0\subseteq Y.$
\end{itemize}
Then, $\{f(u,\cdot)\}_{u\in  U}$  is dense in $\{f(x,\cdot)\}_{x\in K}\subseteq B (Y),$ where $ B (Y)$ is the space of bounded functions on $Y$ endowed with the topology of uniform norm.
\end{theorem}
\begin{proof} Let $k\in K.$ Note at first  that $\inf_{u\in U}f(u,y)\le f(k,y)$ for all $y\in Y.$ Indeed, since $U$ is dense in $K$ for $y\in Y$ fixed, %according to (iii), one has $\inf_{u\in U}f(u,y)=f(u_0,y)$ for some $u_0\in U$ and $f(u_0,y)\le f(u,y)$ for all $u\in U.$
consider a net $(u_i)$ converging to $k.$ Then, from (i) we obtain
$\lim f(u_i,y)\to f(k,y)$, hence $\inf_{u\in U}f(u,y)\le\lim f(u_i,y)\le f(k,y).$

Let $\e>0$ be fixed. Assume that $\sup_{y\in Y}f(k,y)=M$ and $\inf_{y\in Y}f(k,y)=m$ and consider $n\in\N$ such that $\e>\frac{M-m}{n}.$ For every $i\in\{1,2,...,n\}$ consider the set $Y_i=f^{-1}\left(k, \left[m+(i-1)\frac{M-m}{n},m+i\frac{M-m}{n}\right]\right).$
Then from $(ii)$ the set $Y_i$ is closed and convex for all $i\in\{1,2,...,n\}$ and obviously $\cup_{i=1}^n Y_i=Y.$

We apply $(iii)$ for the sets $Y_i,\,i\in\{1,2,...,n\}.$
Hence, $$\inf_{x\in  U}\sup_{y\in Y_i}f(x,y)=\sup_{y\in Y_i}\inf_{x\in U}f(x,y),\mbox{ for all }i\in\{1,2,...,n\}.$$
We have $\sup_{y\in Y_i}\inf_{u\in U}f(u,y)\le \sup_{y\in Y_i}f(k,y)< \e+\inf_{y\in Y_i}f(k,y),\mbox{ for all }i\in\{1,2,...,n\},$ hence
$$\inf_{x\in  U}\sup_{y\in Y_i}f(x,y)< \e+\inf_{y\in Y_i}f(k,y),\mbox{ for all }i\in\{1,2,...,n\}.$$
But  then, for all $i\in\{1,2,...,n\}$ there exists $u_i\in U$ such that
$$\sup_{y\in Y_i}f(u_i,y)< \e+\inf_{y\in Y_i}f(k,y).$$
From the latter relation we get
$$\sup_{y\in Y_i}(f(u_i,y)-f(k,y))\le\sup_{y\in Y_i}f(u_i,y)-\inf_{y\in Y_i}f(k,y)<\e,\mbox{ for all }i\in\{1,2,...,n\}.$$
Obviously for every $y\in Y$ there exists $i\in\{1,2,...,n\}$ such that $y\in Y_i,$
hence
$$\sup_{y\in Y}\inf_{u\in\{u_1,u_2,...,u_n\}}(f(u,y)-f(k,y))<\e.$$
On the other hand $$\inf_{u\in\{u_1,u_2,...,u_n\}}(f(u,y)-f(k,y))\ge\inf_{u\in\co\{u_1,u_2,...,u_n\}\cap U}(f(u,y)-f(k,y)),$$ consequently
$$\sup_{y\in Y}\inf_{u\in\co\{u_1,u_2,...,u_n\}\cap U}(f(u,y)-f(k,y))<\e.$$
We show as in the proof of Theorem \ref{ta}, that Corollary \ref{c4} can be applied to the function $g:K\times Y\To\R,\, g(x,y)=f(x,y)-f(k,y).$

By applying Corollary \ref{c4}, we obtain
$$\inf_{u\in\co\{u_1,...,u_n\}\cap U}\sup_{y\in Y}(f(u,y)-f(k,y))=\sup_{y\in Y}\inf_{u\in\co\{u_1,...,u_n\}\cap U}(f(u,y)-f(k,y))<\e.$$
Hence, there exists $u^*\in\co\{u_1,...,u_n\}\cap U$ such that
$$\sup_{y\in Y}(f(u^*,y)-f(k,y))< \e.$$

Conversely, since $Y$ is symmetric and $f$ is affine in the second variable we have
$$\sup_{y\in Y}(-f(u^*,y)+f(k,y))=\sup_{y\in Y}(f(u^*,-y)-f(k,-y))$$
and $$\sup_{y\in Y}(f(u^*,-y)-f(k,-y))=\sup_{y\in Y}(f(u^*,y)-f(k,y))< \e.$$
Hence, $$\sup_{y\in Y}|f(u^*,y)-f(k,y)|< \e.$$
\end{proof}

%We must emphasize the connection of Theorem \ref{ta1} and some earlier results of Simons \cite{simons1,S,S1}. However, Theorem \ref{ta1} seems to be more general, in the sense that the bifunction involved is not necessarily the duality pairing  and the minimax result holds on a self-segment-dense set.

\begin{remark}\label{r51}\rm  We would like to emphasize that  the self-segment-dense property of $U$ in the hypotheses of Theorem \ref{ta} and  Theorem \ref{ta1} is essential  and cannot be replaced by its denseness. Indeed, let  $X$ be an infinite dimensional real Hilbert space. Let $K=Y=\left\{ x\in X:\left\Vert x\right\Vert \le 1\right\}$ be the unit ball of $X$ and let $U=\left\{ x\in X:\left\Vert x\right\Vert =1\right\}$. Then according to Example \ref{ex1}, $U$ is dense in $K$ with respect to the weak topology of $X,$ but is not self-segment-dense in $K.$ Obviously $K=Y$  is weakly compact. Consider the function $$f:K\times Y\To\R,\,f(x,y)=\<x,y\>.$$ Then, it can easily be verified that the conditions (i) and (ii) in the hypotheses of Theorem \ref{ta} and Theorem \ref{ta1} are fulfilled.
Observe further, that $(iii)$ in the hypothesis of Theorem \ref{ta1} also holds, since  for every weakly closed convex subset $Y_0\subseteq Y$ Theorem \ref{fan}, can be applied for the function $f(x,y)=\<x,y\>.$

Now, the family $\{f(u,\cdot)\}_{u\in  U}$  is not dense in $\{f(x,\cdot)\}_{x\in K}\subseteq C (Y),$ since for $k=0\in K$ one has
$$\sup_{y\in Y}|f(u,y)-f(k,y)|=\sup_{y\in Y}|\<u,y\>|=1,\mbox{ for all }u\in U.$$
\end{remark}
\begin{corollary}\label{c5}  Let $K$ be a  nonempty, compact and convex subset of the Hausdorff locally convex topological vector space $X$ and let $Y$ be a closed convex bounded and symmetric subset of a locally convex topological vector space. Let $U\subseteq K$ be a self-segment-dense set in $K$ and suppose that $\co(U)$ is segment-dense in $K.$ Consider further the bounded mapping $f:K\times Y\To\R,$ and assume  that the following assumptions are fulfilled.
\begin{itemize}
\item[(i)] The map $x\To f(x,y)$ is  convex and continuous on $K$ for all $y\in Y$.
\item[(ii)] The map $y\To f(x,y)$ is affine and continuous on $Y$ for all $x\in  K.$
%\item[(iii)] For every $y\in Y$, $\inf_{x\in U} f(x,y)$ is attained.
%\item[(iii)] The family $(f(\cdot,y))_{y\in Y}$ is an  equicontinuous and pointwise bounded family of $C(K)$.
%\item[(iii)] $\inf_{x\in  U}\sup_{y\in Y_0}f(x,y)=\sup_{y\in Y_0}\inf_{x\in U}f(x,y),$ for every closed convex subset $Y_0\subseteq Y.$
\end{itemize}
Then, $\{f(u,\cdot)\}_{u\in  U}$  is dense in $\{f(x,\cdot)\}_{x\in K}\subseteq B (Y),$ where $ B (Y)$ is the space of bounded functions on $Y$ endowed with the topology of uniform norm.
\end{corollary}
\begin{proof} The conclusion follows via Theorem \ref{tm3}, which assures that the condition (iii) in Theorem \ref{ta1} is satisfied.
\end{proof}

\begin{corollary}{\rm(James' Theorem)} Let $X$ be a Banach space and let $X^*$ be the dual of $X.$ Then $X$ is reflexive if and only if every $x^*\in X^*$ attains its norm on the close unit ball of $X$.
\end{corollary}

\begin{proof}%In other words the condition  every $x^*\in X^*$ attains its norm on the close unit ball of $X$, means that  for every $x^*\in X^*$ there exists $x_0\in B=\{x\in X:\|x\|\le 1\}$ such that $\|x^*\|=\<x^*,x_0\>.$
 %$\|x^*\|=\<x^*,x_0\>$ can be reformulated as $x^*$ takes its supremum (infimum) on the closed unit ball $B=\{x\in X:\|x\|\le 1\}$ of $X.$
The implication $"\Rightarrow"$ is straightforward. Indeed, if $X$ is reflexive, then $B$ is weakly compact, hence by  Weierstrass theorem
$\sup_{x\in B}\<x^*,x\>$ is attained.

 Moreover, for the converse implication it is enough to assume that for all $x^*\in B^*=\{x^*\in X^*:\|x^*\|\le 1\},$ there exists $x_0\in B$ such that $\|x^*\|=\<x^*,x_0\>,$ where $B^*$ is the closed unit ball of $X^*.$ Indeed, for $x^*\in X^*\setminus B^*$ one has $\frac{x^*}{\|x^*\|}\in B^*,$ hence $\left\<\frac{x^*}{\|x^*\|},x_0\right\>=\left\|\frac{x^*}{\|x^*\|}\right\|=1,$ for some $x_0\in B.$

 Let $B^{**}$ be the unit ball of $X^{**}$ and  consider $U=\hat{B}\subseteq X^{**},$ where $\hat{B}$ is the canonical embedding of $B$ in $X^{**}.$ Then, by Goldstine theorem $U$ is dense in $B^{**}$ and  since is convex, $U$ is self-segment-dense in $B^{**}.$
Consider the bifunction $$f:B^{**}\times B^*\To \R,\, f(x^{**},x^*)=\<x^{**},x^{*}\>.$$ From Simons minimax theorem \cite{simons1,S1} it follows that
$$\inf_{u\in U}\sup_{x^{*}\in B^*} f(u,x^*)=\sup_{x^{*}\in B^*}\inf_{u\in U}f(u,x^*).$$
%Since for all $x^*\in B^*$ there exists $x_0\in U$ such that $\|x^*\|=\<x^*,x_0\>,$ one can equivalently say that for every $x^*\in B^*$, $\inf_{u\in U} f(u,x^*)$ is attained.
Hence, by Theorem \ref{ta1} one has that $\{f(u,\cdot)\}_{u\in U}$  is dense in $\{f(x^{**},\cdot)\}_{x^{**}\in B^{**}}\subseteq C(B^*),$ where $ C(B^*)$ is endowed with the topology of uniform norm. In other words, for every $\e>0$ and $x^{**}\in B^{**}$ there exists a $u\in U$ such that
$$\sup_{x^*\in B^*}|f(u,x^*)-f(x^{**},x^*)|<\e.$$
Equivalently,
$$\sup_{x^*\in B^*}|\<u-x^{**},x^{*}\>|=\|u-x^{**}\|<\e.$$
Hence $U$ is dense $B^{**}$ in the strong topology of $X^{**}$. But $\cl U= U$ in the strong topology, hence $U=B^{**}.$ But this implies that $J(X)=X^{**}$, where $J(X)$ is the canonical embedding  of $X$ in $X^{**}$, that is, $X$ is reflexive.
\end{proof}

{\bf Acknowledgements}
 The author would like to thank prof.dr. Radu Ioan Bo\c t, for his valuable suggestions concerning Proposition \ref{p312}.
%This work was supported by a grant of the Romanian Ministry of Education,
%CNCS - UEFISCDI, project number PN-II-RU-PD-2012-3 -0166.\\
%This research was supported by a grant of the Romanian National Authority
%for Scientific Research CNCS - UEFISCDI, project number
%PN-II-ID-PCE-2011-3-0094.

\end{document}